\documentclass[reqno,12pt,letterpaper]{amsart}
\usepackage{amsmath,amssymb,amsthm,graphicx,mathrsfs,url}
\usepackage[usenames,dvipsnames]{color}
\usepackage[colorlinks=true,linkcolor=Red,citecolor=Green]{hyperref}

\def\?[#1]{\textbf{[#1]}\marginpar{\Large{\textbf{??}}}}

\newtheorem{theo}{Theorem}[section]
\newtheorem{prop}[theo]{Proposition}
\newtheorem{lemm}[theo]{Lemma}
\newtheorem{corr}[theo]{Corollary}

\theoremstyle{definition}
\newtheorem{rem}[theo]{Remark}

\numberwithin{equation}{section}

\newcommand{\mc}{\mathcal}
\newcommand{\rr}{\mathbb{R}}
\newcommand{\nn}{\mathbb{N}}
\newcommand{\cc}{\mathbb{C}}

\newcommand{\zz}{\mathbb{Z}}

\newcommand{\la}{\lambda}
\newcommand{\eps}{\epsilon}

\newcommand{\pl}{\partial}
\newcommand{\x}{\times}

\newcommand{\bbar}{\overline}

\newcommand{\cjd}{\rangle}
\newcommand{\cjg}{\langle}

\newcommand{\demi}{\tfrac{1}{2}}

\DeclareMathOperator{\rank}{rank}

\title{Invariant distributions and X-ray transform for Anosov flows}
\author{Colin Guillarmou}
\email{cguillar@dma.ens.fr}
\address{DMA, U.M.R. 8553 CNRS, \'Ecole Normale Superieure, 45 rue d'Ulm,
75230 Paris cedex 05, France}

\begin{document}
\begin{abstract}
For Anosov  flows preserving a smooth measure on a closed manifold $\mc{M}$, we define a natural self-adjoint operator $\Pi$  which maps into the space of flow invariant distributions in $\cap_{r<0} H^{r}(\mc{M})$ and whose kernel is made of coboundaries in $\cup_{s>0} H^{s}(\mc{M})$. We describe relations to the Livsic theorem and recover regularity properties of cohomological equations using this operator. For Anosov geodesic flows on the unit tangent bundle $\mc{M}=SM$ of a compact manifold $M$, we apply this theory to study  $X$-ray transform on symmetric tensors on $M$. In particular we prove existence of flow invariant distributions on $SM$ with prescribed push-forward on $M$ and a similar version for tensors. This allows us to show injectivity of the X-ray transform on an Anosov surface: any divergence-free symmetric tensor on $M$ which integrates to $0$ along all closed geodesics is zero.
\end{abstract}

\maketitle

\section{Introduction}
In this work, we develop new approaches and tools coming from the theory of anisotropic Sobolev spaces for studying cohomological equations, invariant distributions and X-ray transform for Anosov flows preserving a Lebesgue type measure.

Ergodicity of the geodesic flow on a closed manifold $M$ can be described by saying that the only invariant $L^2(SM)$ functions are the constants. By contrast, if the flow is Anosov, there exist infinitely many invariant distributions, as one can take Dirac measures supported on closed geodesic. These distributions are not very regular in terms of Sobolev regularity, and it is natural to ask if there exist invariant distributions which are in Sobolev spaces $H^{r}(SM)$ for small $r<0$. In connection to tensor tomography, it was recently shown by Paternain-Salo-Uhlmann \cite{PSU1,PSU2} that there exist infinitely many invariant distributions which are in $H^{-1}(SM)$ for Anosov geodesic flows. More precisely, if one fixes $f\in L^2(M)$ with $\int_Mfdv_g=0$, then they show that there is an invariant distribution $w\in 
H^{-1}(SM)$ such that ${\pi_0}_*w=f$ if $\pi_0:SM\to M$ is the projection on the base. Their work is based on  
Fourier decomposition in the fibers of $SM$ as initiated by Guillemin-Kazhdan \cite{GK1,GK2}. For hyperbolic surfaces, certain invariant distributions 
in $H^{-1/2-\eps}(SM)$ for all $\eps>0$ are studied by Anantharaman-Zelditch \cite{AnZe} in connection with quantum ergodicity.

In the present work, we address several related questions. Recall that a flow generated by a vector field $X$ on a compact manifold $\mc{M}$ is Anosov if there is a continuous splitting $T\mc{M}=\rr X\oplus E_u\oplus E_s$ where $E_s$ and $E_u$ are stable and unstable spaces (see \eqref{stabunstab}). The dual bundles $E_u^*$ and $E_s^*$ are defined by $E_u^*(E_u\oplus \rr X)=0$ and $E_s^*(E_s\oplus \rr X)=0$.
Our first result is 
\begin{theo}\label{th0}
Let $\mc{M}$ be a compact manifold and $X$ be a smooth vector field generating an Anosov flow preserving a smooth invariant probability measure $d\mu$, and assume the flow is mixing for this measure. Then there exists a bounded operator 
\[\Pi:H^{s}(\mc{M})\to H^{r}(\mc{M}) ,\quad \forall s>0, \,\, \forall r<0,\] 
with infinite dimensional range, dense in the space $\mc{I}:=\{w\in C^{-\infty}(\mc{M}), Xw=0\}$ of invariant distributions, that is self-adjoint as a map $H^{s}(\mc{M})\to H^{-s}(\mc{M})$ for any $s>0$ and 
satisfies\footnote{WF denotes wave front set of distributions, cf. \cite[Chapter 8.1]{Ho1}} 
\[\begin{gathered}
X\Pi f=0,\,\,\,  \forall f\in H^s(\mc{M}), \,\, \textrm{ and } \,\, \Pi Xf=0,\,\,\, \forall f\in H^{s+1}(\mc{M}),\\
f\in C^\infty(\mc{M})\Longrightarrow {\rm WF}(\Pi f)\subset E_u^*\cup E_s^*.
\end{gathered}\]
If $f\in H^{s}(\mc{M})$ with $\cjg f,1\cjd_{L^2}=0$, then $f\in \ker\Pi$ if and only if there exists $0<s'\leq s$ and a  solution $u\in H^{s'}(\mc{M})$ to $Xu=f$, in which case $u$ is actually in $H^{s}(\mc{M})$ and is the only solution 
in $L^2$ modulo constants. 
\end{theo}
The result also holds with a weaker assumption than mixing, see Theorem \ref{thPi}. 
This result shows that the ergodicity statement in terms of unique $L^2$-invariant distribution cannot be extended to less regular Sobolev spaces, and that the space of invariant distributions in $\cap_{r<0}H^{r}(\mc{M})$ is infinite dimensional. 
To define the operator $\Pi$, we use the recent work of Faure-Sj\"ostrand \cite{FaSj} on the Ruelle spectrum of Anosov flows, which involves anisotropic Sobolev spaces associated to the hyperbolic dynamic; such spaces appeared first in \cite{BKL,Li,GoLi,BaTs,BuLi} to study speed of mixing of Anosov diffeomorphisms and flows. In fact, the operator can also be defined as a weak limit of damped correlations 
\begin{equation}\label{weaklim} 
\cjg \Pi f,\psi\cjd =\lim_{\la\to 0^+}\int_{\rr}e^{-\la |t|}\cjg f\circ \varphi_t,\psi\cjd dt, \quad 
f,\psi\in C^\infty(\mc{M})
\end{equation}
if $\int_{\mc M} f d\mu=0$.
The microlocal structure (wave front set) of $\Pi$ follows directly from the recent work of Dyatlov-Zworski \cite{DyZw1}, see Proposition \ref{dyatlovzworski} below.
As a corollary, we also recover smoothness result for 
the Livsic cohomological equation in our setting:
\begin{corr}\label{corlivsic}
With the same assumptions as in Theorem \ref{th0}, for all $s>\dim(\mc{M})/2$ and
$f\in H^{s}(\mc{M})$ satisfying $\int_{\gamma}f=0$ for all closed orbit $\gamma$ of $X$, 
there exist $u\in H^{s}(\mc{M})$ such that $Xu=f$. In particular if $f\in C^\infty(\mc{M})$ then $u\in C^\infty(\mc{M})$.
\end{corr}
De la Llave-Marco-Moriyon \cite{DMM}, Journ\'e \cite{Jo} or Hurder-Katok \cite{HuKa}
proved a version of Corollary \ref{corlivsic} with a regularity statement in H\"older spaces $C^{k,\alpha}$ (the case of geodesic flows with negative curvature was done previously in \cite{GK1,GK2}), but
Theorem \ref{th0} gives a Sobolev regularity version of Livsic theorem, which did not seem to be known. In particular, here we even get a characterization of coboundaries in $H^s$ for any $s>0$, including cases of unbounded functions (if $0<s<n/2$), which is not possible by the usual Livsic theorem.

Next, we focus on the setting of Anosov geodesic flows, where there is an additional structure.
In this case, $\mc{M}=SM$ is the unit tangent bundle of a compact manifold, which is a sphere bundle with natural projection $\pi_0:SM\to M$. In the case of manifolds with boundary, we can define the  X-ray transform $I_0$ on functions when the metric is simple (i.e.\ the boundary is strictly convex boundary and the exponential map is a diffeomorphism at any point): this is defined as a map $I_0: C^\infty(M)\to C^{\infty}(\pl_-SM)$ where $\pl_-SM$ is the inward pointing boundary of $SM$, and $I_0f(x,v)=\int_{0}^{\ell(x,v)}f(\pi_0(\varphi_t(x,v)))dt$ is  the integral of $f$ along the geodesic in $M$ with starting point $x$ and tangent vector $v$ and $\ell(x,v)$ is the length of that geodesic. The question of injectivity of $I_0$ and surjectivity of $I_0^*$ are important in the study of inverse problems and tomography, and the description of the operator $I_0^*I_0$ as an elliptic pseudo-differential operator has been fundamental to solve these questions, see for instance \cite{PeUh,UhVa}.
In the Anosov setting  (the manifold does not have boundary), we define  
 \[\Pi_0 = {\pi_0}_*\Pi\, \pi_0^* : C^\infty(M)\to C^{-\infty}(M).\]
This operator will play the role of $I_0^*I_0$ when there is no boundary. We show 
\begin{theo}\label{Th2}
Let $(M,g)$ be a compact manifold with Anosov geodesic flow. Then\\ 
1) The operator $\Pi_0$ is an elliptic pseudo-differential operator of order $-1$, with principal symbol $\sigma(\Pi_0)(x,\xi)=
C_n|\xi|_g^{-1}$ for some $C_n\not=0$ depending only on $n$. \\
2) The operator $\Pi\, \pi_0^*$ is well-defined on $H^{s}(M)$ for all $s\in \rr$ and is injective.\\
3) Let $s\leq 1$, then for any $f\in H^{s}(M)$, there exists $w\in C^{-\infty}(SM)$  
so that $Xw=0$ and ${\pi_0}_*w=f$ and $w$ has the regularity 
\[ w\in \bigcap_{r<0}H^{r}(SM) \textrm{ if }s=1 \textrm{ and }  w\in H^{s-1}(SM), \textrm{ if }s<1.\] 
%If $f\in C^\infty(SM)$ then ${\rm WF}(w)\subset E_u^*\cup E_s^*$.
\end{theo} 
If we were working in the setting of simple manifolds (case with boundary),  3) would provide surjectivity for the operator $I_0^*$. The result in 3) improves the recent work \cite{PSU2} in terms of regularity and gives in addition a precise description of the singularities of the invariant distribution (through the wave front set). Finally we describe tomography for $m$-cotensors: symmetric cotensors $f\in C^\infty(M,\otimes_S^mT^*M)$ can be mapped to smooth functions on $SM$ via the map $(\pi_m^*f)(x,v):= \cjg f(x),\otimes^m v\cjd$, and 
 denote by ${\pi_m}_*: C^{-\infty}(SM)\to C^{-\infty}(M,\otimes_S^mT^*M)$ its adjoint acting on distributions.
The divergence of $m$-cotensors is defined by $D^*f:=-\mc{T}(\nabla f)$ where $\nabla$ is the Levi-Civita connection and $\mc{T}$ is the trace defined in \eqref{deftrace}. 
As for $m=0$, we show in Theorem \ref{microlocalPim} and Corollary \ref{corsurjm} a result similar to Theorem \ref{Th2} for the operator $\Pi_m:={\pi_m}_*\Pi\, \pi_m^*$; for instance 
it is a pseudo-differential operator of order $-1$ which is elliptic on the space of divergence-free cotensors. We also show the existence of invariant distributions $w$ with regularity as in 3) with prescribed divergence-free value ${\pi_m}_*w$, assuming injectivity of $\Pi\, \pi_m^*|_{\ker D^*}$. We remark that this injectivity 
follows from the work of Croke-Sharafutdinov \cite{CS} if the sectional curvatures of $(M,g)$ are non-positive; 
in the Anosov setting, the kernel is trivial for $m=0,1$ as a consequence of the work of Dairbekov-Sharafutdinov \cite{DaSh}.

To conclude, we give applications to the injectivity of the X-ray transform on Anosov surfaces. 
Let $(M,g)$ be a Riemannian surface and let $\varphi_t:SM\to SM$ be its geodesic flow on the unit tangent bundle $SM$ of $M$.  
A closed geodesic is a curve $\gamma$ on $M$ such that there exists $\ell>0$ and 
$\gamma:=\{\pi_0(\varphi_t(x,v))\in M; t\in[0,\ell], \varphi_\ell(x,v)=(x,v)\}$; the smallest such $\ell$ is denoted $\ell_\gamma$ and 
called the length of $\gamma$. A parametrization of $\gamma$ is given by 
$\gamma(t)=\pi_0(\varphi_t(x,v))$ if $t\in [0,\ell_\gamma]$ and $x\in \gamma$ and $v\in S_xM$ is tangent to 
$\gamma$. 
The set of closed geodesics of $M$ is countable and denoted by $\mc{G}$. We define the X-ray transform on symmetric $m$-cotensors  
as the linear map 
\[\begin{gathered} 
I_m:  C^{\infty}(M,\otimes_S^mT^*M)\to (\mc{G}\to \rr), \quad 
I_m(f)(\gamma):= \int_{0}^{\ell_\gamma}\cjg f(\gamma(t)),\otimes^m \dot\gamma(t)\cjd dt   
\end{gathered}\] 
where dot denotes the time derivative. If $D$ denotes symmetrized covariant derivative, we remark 
that any $f$ which is written under the form $f=Dh$ for some 
$h\in C^{\infty}(M,\otimes_S^{m-1}T^*M)$ satisfies $I_m(f)(\gamma)=0$ for all $\gamma\in\mc{G}$ and thus 
the kernel $\ker I_m$ is infinite dimensional for $m\geq 1$. It is then natural to consider $I_m$ acting on divergence-free 
cotensors.  
In negative curvature, it has been proved in \cite{CS} that $\ker I_m\cap \ker D^*=0$ (this was first shown by \cite{GK1} for negatively curved surfaces), then the proof that $I_0$ is injective and 
that $\ker I_1\cap \ker D^*=0$ for Anosov manifolds appeared in \cite{DaSh}. 
More recently, Paternain-Salo-Uhlmann \cite{PSU1} proved that $\ker I_2\cap \ker D^*=0$ for Anosov surfaces.
Here, we prove injectivity of X-ray transform on all divergence-free symmetric cotensors for Anosov surfaces, which has been an important open problem in the field (see for example \cite[Problem 1.7]{DaSh}).
\begin{theo}\label{injectm}
On  a Riemannian surface with Anosov geodesic flow, then for all $m\geq 0$ we have 
$\ker I_m\cap \ker D^*=0$ and $\ker \Pi \, \pi_m^*\cap \ker D^*=0$.
\end{theo}
This is the exact analogue in the Anosov case of the recent result of Paternain-Salo-Uhlmann \cite{PSU3}
for simple domains in dimension $2$, but both the proof and the geometric setting are very different; it 
is clear that the method of \cite{PSU3} uses strongly that the manifold is open and non-trapping.
To prove Theorem \ref{injectm}, we use the existence of invariant distributions with prescribed push-forward for $m=1$, we apply a  Szeg\" o projector in the fibers of $SM$ to these distributions, which allows us to multiply them using their wave front set property. The obtained set of distributions produce a large enough vector space of invariant distributions orthogonal to $\ker I_m\cap \ker D^*$ to force this space being trivial. The properties of our operator $\Pi$ allows to solve the problem encountered in Remark 9.4 of \cite{PSU1}, which was why they had to assume $m\leq 2$.\\

To motivate even more this new way of analyzing X-ray transform, we mention that 
using tools of similar nature, we are able to prove in \cite{Gu} the injectivity of X-ray transform for all negatively curved manifold with strictly convex boundary, and that the scattering map for the flow on negatively curved surfaces determines the Riemannian metric up to conformal diffeomorphism.\\

\noindent \textbf{Ackowledgement.} We thank N. Anantharaman, V. Baladi, S. Dyatlov, F. Naud, G. Paternain, M. Salo, G. Uhlmann, M. Zworski for useful discussions. The research is partially supported by grants ANR-13-BS01-0007-01 and ANR-13-JS01-0006.

\section{The resolvent of the flow and the operator $\Pi$}

In what follows, the manifold $\mc{M}$ will be connected. Before we start, we point out that we shall use 
pseudo-differential operators ($\Psi$DO) and the notion of wave front set of distributions and of $\Psi$DOs. 
We refer to H\"ormander \cite[Chap. VIII]{Ho1} for wave front sets, and to Grigis-Sj\"ostrand \cite{GrSj} for standard pseudo-differential calculus and Zworski \cite{Zw} for the semiclassical version. We shall say that a pseudo-differential operator $A$ on $\mc{M}$ is microsupported in a conic set $U\subset T^*\mc{M}$ if its wave front set is contained in 
$U$, ie. its full symbol $a(y,\xi)$ in local coordinates and all its derivatives vanish to all order as $|\xi|\to \infty$ outside the conic set $U$.  We will write $\Psi^{m}(\mc{M})$ for the class of classical  pseudo-differential 
operators of order $m\in \rr$. We say that $A\in \Psi^m(\mc{M})$ is elliptic at $(x_0,\xi_0)$ if there is a conical neighborhood $U$ of $(x_0,\xi_0)$ and $C>0$ such that 
its principal symbol $\sigma(A)$ satisfies $|\sigma(A)(x,\xi)|/|\xi|^m\geq 1/C$ in $U\cap \{|\xi|>C\}$.

\subsection{The resolvent of the flow}
First, we recall that a smooth vector field $X$ on a compact manifold $\mc{M}$ without boundary is Anosov if its 
flow $\varphi_t$ has the following property: there exists a continuous flow-invariant splitting 
\[ T\mc{M}=\rr X\oplus E_s\oplus E_u\]
where $E_s, E_u$ are the stable/unstable bundles, which are defined as follows: 
there exists $C>0,\nu>0$ such that 
\begin{equation}\label{stabunstab}
\begin{gathered}
\xi \in E_s(y), y\in \mc{M} \iff \|d\varphi_t(y)(\xi)\|\leq Ce^{-\nu t}\|\xi\| \quad \textrm{ for }t\geq 0 \\
\xi \in E_u(y), y\in \mc{M} \iff \|d\varphi_t(y)(\xi)\|\leq Ce^{-\nu |t|}\|\xi\| \quad \textrm{ for }t\leq 0
\end{gathered}\end{equation}
where $\|\cdot\|$ is the norm induced by any fixed metric on $\mc{M}$.
The flow will be said to be a contact Anosov flow if the Anosov form $\alpha$, defined by 
\[ \alpha(X)=1 ,\quad \ker \alpha =E_u\oplus E_s,\]
is a smooth contact form (i.e.\  $d\alpha$ is symplectic on $\ker \alpha$). 
This is for instance the case for Anosov geodesic flows on $\mc{M}=SM$ with 
$(M,g)$ a Riemannian compact manifold, as $\alpha$ is simply the Liouville $1$-form.
Notice that for contact Anosov flow, there is a natural invariant measure given by 
$d\mu=\alpha\wedge (d\alpha)^d$ where $d$ is the dimension of $E_u$ (and $E_s$).
We shall also define the dual stable and unstable bundles $E_s^*\subset T^*\mc{M}$ and $E_u^*\subset T^*\mc{M}$ by 
\begin{equation}\label{Eus*} 
E_u^*(E_u\oplus \rr X)=0, \quad E_s^*(E_s\oplus \rr X)=0.
\end{equation}

When the flow $\varphi_t$ has a smooth invariant measure $d\mu$ on $\mc{M}$, 
the generating vector field can be viewed as an (formally) anti self-adjoint operator on $C^\infty(\mc{M})$, that is 
$\cjg Xu,v\cjd=-\cjg u,Xv\cjd$ for all $u,v\in C^\infty(\mc{M})$ where the pairing is the $L^2(\mc{M})$ pairing using the invariant measure $d\mu$. By Stone's theorem, the generator $-iX$ of the unitary operator 
\[e^{tX}:L^2(\mc{M})\to L^2(\mc{M}), \quad (e^{tX}f)(y)=f(\varphi_t(y))\]
is self-adjoint on $L^2(\mc{M})=L^2(\mc{M},d\mu)$.
In this case, the spectral theorem for self-adjoint operators tells us that ${\rm Spec}_{L^2}(-iX)\subset \rr$ .
The resolvents $R_-(\la)=(-X-\la)^{-1}$ and $R_+(\la)=(-X+\la)^{-1}$ are well-defined on $L^2(\mc{M})$ for ${\rm Re}(\la)>0$  by 
\begin{equation}\label{R_+R_-}
\begin{gathered} 
R_+(\la)f(y)=\int_0^\infty e^{-\la t}f(\varphi_t(y))dt, \quad 
R_-(\la)f(y)=-\int_{-\infty}^0 e^{\la t}f(\varphi_t(y))dt.
\end{gathered}
\end{equation}
Moreover, Stone's formula gives the spectral measure of $-iX$ in terms of the resolvents $R_\pm(i\la)$ by the strong limit  
\begin{equation}\label{Stone} 
\demi(1_{[a,b]}(-iX)+1_{(a,b)}(-iX))=(2\pi)^{-1}\lim_{\eps\to 0^+}\int_{a}^b(R_+(i\la+\eps)-R_-(-i\la+\eps))d\la
\end{equation}
for $a,b\in\rr$.
Then we recall the result of Faure-Sj\"ostrand \cite[Th. 1.4]{FaSj}:
\begin{theo}[Faure-Sj\"ostrand]\label{fauresjos}
Assume that $X$ is a smooth vector field generating an Anosov flow and let $d\mu$ be a smooth invariant measure.
 There exists $c>0$ such that for all $s>0$ and $r<0$, there is a Hilbert space $\mc{H}^{r,s}$ such that $H^{s}(\mc{M})\subset \mc{H}^{r,s}\subset H^{r}(\mc{M})$ and 
$-X-\la$ is Fredholm with index $0$ as an operator 
\[-X-\la : {\rm Dom}(X)\cap\mc{H}^{r,s}\to \mc{H}^{r,s}, \quad  \textrm{ if }{\rm Re}(\la)> -c\min(|r|,s)\] 
depending analytically on $\la$. Moreover $-X-\la$ is invertible for ${\rm Re}(\la)$ large enough on these spaces, the inverse coincides with $R_-(\la)$ when acting on $H^{s}(\mc{M})$ and it extends meromorphically to the half-plane ${\rm Re}(\la)> -c\min(|r|,s)$, with poles of finite multiplicity as a bounded operator on $\mc{H}^{r,s}$.
\end{theo}
The assumption about the invariant smooth measure $d\mu$ is not necessary in \cite{FaSj}, but without it, the Fredholm property 
is true only in ${\rm Re}(\la)>-c\min(|r|,s)+\la_0$ for some $\la_0$ so that $-X-\la$ be invertible on 
$L^2(\mc{M})$ (defined with respect to some fixed smooth measure) for ${\rm Re}(\la)>\la_0$. 
If the flow is contact, there is an invariant smooth measure making $X$ anti-self adjoint 
and thus one can take $\la_0=0$, with inverse given by  \eqref{R_+R_-}. For what follows, we shall assume, without loss of generality, that the invariant measure is a probability measure.

The fact that the inverse $(-X- \la)^{-1}$ for ${\rm Re}(\la)$ large coincides with $R_- (\la)$ when acting on $H^s(\mc{M})$ is not explicitly stated in \cite{FaSj} but it is straightforward to check, since by formula \eqref{R_+R_-}, $R_-(\la)$ maps $H^s(\mc{M})$ to itself for ${\rm Re}(\la)>C|s|$ for some $C>0$ depending only on the Lyapunov exponents of $\varphi_t$ (see \cite[Prop. 3.2]{DyZw1} for details).
As operator mapping $C^\infty(\mc{M})\to C^{-\infty}(\mc{M})$, the resolvent $R_-(\la)$ does not depend on $r,s$;
the Schwartz kernel of $R_-(\la)$ admits a meromorphic extension to $\cc$ as an element in $C^{-\infty}(\mc{M}\x\mc{M})$. 
Since moreover $(-X-\la)^*=(X-\bbar{\la})=-(-X+\bbar{\la})$ on $C^\infty(\mc{M})$, we have that for ${\rm Re}(\la)>0$ 
\begin{equation}\label{adjoint} 
R_-(\bbar{\la})^*= - R_+(\la)  \textrm{ on }L^2(\mc{M}). 
\end{equation}

\textbf{The anisotropic Sobolev spaces.} In \cite[Sec. 1.1.2]{FaSj}, the spaces $\mc{H}^{r,s}$ are defined by $\mc{H}^{r,s}:=\hat{A}_{r,s}^{-1}(L^2(\mc{M}))$ where $\hat{A}_{r,s}={\rm Op}(A_{r,s})$ is an invertible
pseudo-differential operator in an anisotropic class which quantizes a symbol function 
$A_{r,s}\in C^\infty(T^*\mc{M})$. We recall the construction of $A_{r,s}$ following \cite{FaSj}.
The symbol is defined by $A_{r,s}=\exp(G_m)$ where $G_m$ is a smooth function on $T^*\mc{M}$ constructed in Lemma 1.2 of \cite{FaSj}: $G_m(y,\xi):=m(y,\xi)\log(\sqrt{1+f(y,\xi)^2})$ with $m$ and $f$ smooth, homogeneous in $|\xi|>1$ with respective degree $0$ and $1$, $f$ is positive and does not depend on $r,s$. 
Fix $\eps>0$ small, the function $m$ is given in $\{|\xi|>1\}$ by\footnote{We choose $n_0:=(1-\eps)s$ in the notation of \cite[Lemma 1.2]{FaSj}.}
$m=s(1-m_2-\eps (m_1-m_2))+rm_2$ where $m_1\in C^\infty(S^*\mc{M};[0,1])$ is such that  $m_1^{-1}(0)$ is a small conic neighborhood of $E_s^*$ and $m_1^{-1}(1)$ is a small conic neighborhood of $E_u^*\oplus E_0^*$, 
and $m_2\in C^\infty(S^*\mc{M};[0,1])$ is such that $m_2^{-1}(1)$ is a small conic neighborhood of $E_u^*$ and $m_2^{-1}(0)$ is a small conic neighborhood of $E_s^*\oplus E_0^*$. The conic neighborhoods can be taken as small as we like in the construction. From the construction, we have  
\begin{equation}\label{nested}
H^s(\mc{M}) \subset \mc{H}^{r,s}\subset \mc{H}^{r,s'}\subset \mc{H}^{r',s'}\subset H^{r'}(\mc{M})
\end{equation}
if $r'\leq r <0<s'\leq s$. We also see that $m=s$ in a small conic neighborhood of $E_s^*$ and 
$m=r$ in a small conic neighborhood of $E_u^*$, which implies that the space $\mc{H}^{r,s}$ is microlocally equivalent 
to $H^s(\mc{M})$ in a small conic open neighborhood $V_s$ of $E_s^*$ in the following sense:  
if $B\in \Psi^0(\mc{M})$ has wave front set contained in $V_s$, then there is $C>0$ such that for all $f\in C^\infty(\mc{M})$
\[ ||Bf||_{H^s(\mc{M})}\leq C||f||_{\mc{H}^{r,s}},\quad ||Bf||_{\mc{H}^{r,s}}\leq C||f||_{H^s(\mc{M})}.\]
If we choose $\eps>0$ small, we get that 
in the region $W_s^\eps:=\{m_2\leq 1-2\eps\}$, $m\geq s\eps+r$. Thus for any  $B\in \Psi^0(\mc{M})$ with wave front set contained in $W_s^\eps$, 
\begin{equation}\label{fBf}
f\in \mc{H}^{r,s} \Longrightarrow Bf\in H^{\eps s+r}(\mc{M}).
\end{equation}
The complement of $W_s^\eps$ is a small conic neighborhood of $E_u^*$ if $\eps$ is small, which means that 
if we choose $s$ and $r$ so that $s\eps+r>0$, functions in $\mc{H}^{r,s}$ are microlocally in a positive Sobolev space outside a small conic neighborhood of $E_u^*$, which can be made as small as we want in the construction.
The dual space $(\mc{H}^{r,s})^*$ to $\mc{H}^{r,s}$ (with respect to $L^2$-pairing)
 is identified with $\hat{A}_{r,s}(L^2(\mc{M}))$ and the symbol of $\hat{A}_{r,s}^{-1}$ being $\exp(-G_m)$, 
 this space is microlocally equivalent to $H^{-r}(\mc{M})$
in an  small conic neighborhood of $E_u^*$ and is microlocally equivalent to $H^{-s}(\mc{M})$ in a small conic neighborhood of $E_s^*$. 

Considering the flow in backward time, which amounts to consider the generator $-X$ instead of $X$, the roles of 
$E_s$ and $E_u$ are exchanged and we can define the space $\mc{H}^{s,r}$ for $r<0<s$ just as $\mc{H}^{r,s}$ but exchanging $E_u^*$ and $E_s^*$: $\mc{H}^{s,r}$ is microlocally equivalent to $H^s(\mc{M})$ in a small conic neighborhood of $E_u^*$, and is microlocally equivalent to $H^r(\mc{M})$ in a small conic neighborhood of $E_s^*$. Like \eqref{nested}, we have for $r'<r<0<s'<s$
 \begin{equation}\label{nested2}
H^s(\mc{M}) \subset \mc{H}^{s,r}\subset \mc{H}^{s',r}\subset \mc{H}^{s',r'}\subset H^{r'}(\mc{M}).
\end{equation}
We then deduce from Theorem \ref{fauresjos} applied with $-X$ that 
 $X-\la$ is Fredholm with index $0$ as an operator 
\[X-\la : {\rm Dom}(X)\cap\mc{H}^{s,r}\to \mc{H}^{s,r}, \quad  \textrm{ if }{\rm Re}(\la)> -c\min(|r|,s)\] 
depending analytically on $\la$. Moreover $X-\la$ is invertible for ${\rm Re}(\la)$ large enough on these spaces, the inverse coincides with $-R_+(\la)$ when acting on $H^{s}(\mc{M})$ and it extends meromorphically to the half-plane ${\rm Re}(\la)> -c\min(|r|,s)$, with poles of finite multiplicity as a bounded operator 
\[ R_+(\la): \mc{H}^{s,r}\to \mc{H}^{s,r}.\]
Notice that the formula \eqref{adjoint}  relating $R_+(\la)$ with $R_-(\la)$ then extends meromorphically to the half-plane ${\rm Re}(\la)>-c\min(|r|,s)$ as an operator  $H^s(\mc{M})\to H^{-s}(\mc{M})$, and shows that 
\begin{equation}\label{dual}
R_-(\la): (\mc{H}^{s,r})^*\to (\mc{H}^{s,r})^*, \quad R_+(\la): (\mc{H}^{r,s})^*\to (\mc{H}^{r,s})^*
\end{equation}
are bounded for $\la$ in the same half-plane.
As above in \eqref{fBf}, there is a small conic neighborhood (that can be made as small as we like) with complement $W_u^\eps$ such that for any $B\in \Psi^0(\mc{M})$ with wave front set contained in $W_u^\eps$, 
\begin{equation}\label{fBf2}
f\in \mc{H}^{s,r} \Longrightarrow Bf\in H^{\eps s+r}(\mc{M}).
\end{equation}
By choosing $\eps>0$ small enough so that  
$W_u^\eps\cup W_s^\eps$ contains $\{|\xi|\geq 1\}$, we deduce from \eqref{fBf} and \eqref{fBf2}, the following regularity statement:
\begin{lemm}\label{regul}
For $s>0$ and $r<0$, we have $\mc{H}^{r,s}\cap \mc{H}^{s,r}\subset H^{s\eps+r}(\mc{M})$ for some 
$\eps>0$ independent of $s,r$.
\end{lemm}

\textbf{Wave front set of the resolvent.} The wave front set of the Schwartz kernel of the resolvent is analyzed by Dyatlov-Zworski \cite[Prop. 3.3]{DyZw1}.
\begin{prop}[Dyatlov-Zworski]\label{dyatlovzworski}
Let $\la_0\in \cc$, and assume that the meromorphically extended resolvent $R_-(\la)$ has a pole of order $k$ at $\la_0$ with Laurent expansion
\[R_-(\la)=R_-^{\rm hol}(\la) +\sum_{j=1}^k \frac{A_j}{(\la-\la_0)^j}\] 
where $R_-^{\rm hol}(\la)$ is holomorphic near $\la_0$, then the wave front set of the Schwartz kernel of $R_-^{\rm hol}(\la)$ satisfies
\[ {\rm WF}(R_-^{\rm hol}(\la))\subset N^*\Delta(\mc{M}\x\mc{M})\cup \Omega_+\cup (E_u^*\x E_s^*).\] 
with $N^*\Delta(\mc{M}\x\mc{M})$ the conormal bundle to the diagonal $\Delta(\mc{M}\x\mc{M})$ of $\mc{M}\x\mc{M}$ and 
\[\Omega_+:= \{(\Phi_t(y,\xi),y,-\xi)\in T^*(\mc{M}\x\mc{M});  \,\, t\geq 0\,\, \xi(X(y))=0\}.\]
where $\Phi_t$ is the symplectic lift of $\varphi_t$ on $T^*\mc{M}$, or equivalently the Hamilton flow of the Hamiltonian $p(y,\xi)=\xi(X(y))$.
A similar result holds for $R_+(\la)$, where the wave front set of the regular part $R_+^{\rm hol}(\la)$ satisfies
\[ {\rm WF}(R_+^{\rm hol}(\la))\subset N^*\Delta(\mc{M}\x\mc{M})\cup \Omega_-\cup (E_s^*\x E_u^*)\] 
where $\Omega_-$ is defined like $\Omega_+$ but with $\Phi_{-t}$ instead of $\Phi_t$.
\end{prop}
We recall that the symplectic lift $\Phi_t$ acts by $\Phi_t(y,\xi)=(\varphi_t(y),(d\varphi_t^{-1}(y))^T\xi)$.

\textbf{Poles on the critical line}. To end this section, we describe the poles of $R_\pm(\la)$ on the imaginary line.
\begin{lemm}\label{poles}
The resolvent $R_\pm(\la)$ have poles of order at most $1$ on the line $i\rr$, and 
we have for all $\la_0\in \rr$
\[ \demi(({\rm Res}_{\la_0}R_+(\la))+({\rm Res}_{\la_0}R_+(\la))^*)=\Pi_{{\rm Eig}(\la_0)}\]
where $\Pi_{{\rm Eig}(\la_0)}$ is the orthogonal projector on the eigenspace ${\rm Eig}(\la_0)=\ker_{L^2}(-iX-\la_0)$.
\end{lemm}
\begin{proof} First we notice that from the spectral theorem, if $i\la_0\in i\rr$ is a pole of $R_\pm(\la)$, it is a first order pole since there is $C>0$ such that for all $u,v\in C^\infty(\mc{M})$ 
\[ |\cjg R_\pm(\la)u,v\cjd|\leq C|\la-i\la_0|^{-1}||u||_{L^2}||v||_{L^2} \quad \textrm{ if }{\rm Re}(\la)>0.\]
We write the Laurent expansion of $R_\pm(\la)$ at $i\la_0$ using \eqref{adjoint}: for $\la,\la_0\in\rr$ and $\eps>0$ 
\[ R_+(i\la+\eps)=\frac{A_0}{i\la+\eps-i\la_0}+\mc{O}(1), \quad -R_-(-i\la+\eps)=R_+(i\la+\eps)^*=\frac{A_0^*}{-i\la+\eps+i\la_0}+\mc{O}(1)
\]
for some operator $A_0$ as $\la\to \la_0$ (the $\mc{O}(1)$ is in the weak sense when applying the identity to 
$f\in C^\infty(\mc{M})$ and pairing with $\psi\in C^\infty(\mc{M})$). 
Thus we get from \eqref{Stone}
\[\begin{gathered}
\demi(1_{[\la_0-\delta,\la_0+\delta]}(-iX)+1_{(\la_0-\delta,\la_0+\delta)}(-iX))=\\
\frac{A_0+A_0^*}{2\pi} \lim_{\eps\to 0^+}
\int_{-\delta}^\delta \frac{1}{i\la+\eps}d\la+\mc{O}(\delta)=\demi(A_0+A_0^*)+\mc{O}(\delta)
\end{gathered}\]
where the $\mc{O}(\delta)$ was independent of $\eps$ (and as above is in weak sense). 
Then letting $\delta\to 0$ we get the result.
\end{proof}

\subsection{Mixing}
We say that a flow $\varphi_t$ is mixing with respect to an invariant probability measure $d\mu$ if for all $u,v\in L^2(\mc{M})$  
\[ C_t(u,v):=\int_{\mc{M}}u(\varphi_t(y)) v(y) d\mu(y) - \int_{\mc{M}} u(y)d\mu(y) \int_{\mc{M}} v(y)d\mu(y) \] 
tends to $0$ as $t\to\infty$.
\begin{lemm}\label{mixing}
Let $X$ be a smooth Anosov vector field on a compact manifold $\mc{M}$ and let $d\mu$ be an invariant measure with respect to 
the flow of $X$. Then the flow is mixing if and only if the only pole of $R_\pm(\la)$ on the line $i\rr$ is $\la=0$ and it is a simple pole with residue $\pm (1\otimes 1)$. 
%In this case the rate of mixing for $u,v\in C^\infty(\mc{M})$  is $C_t(u,v)=\mc{O}(t^{-\infty})$.
\end{lemm}
\begin{proof} 
Assume that the flow is mixing. Let $u,v\in C^\infty(\mc{M})$, then for all $\eps>0$ small, there is $T_\eps$ such that for $|t|>T_\eps$, 
$|C_{t}(u,v)|\leq \eps$. Using \eqref{R_+R_-} we have for $\la>0$
\[\begin{split} 
\la\cjg R_+(\la)u,v\cjd=&\int_{0}^{T_\eps}\!\! \int_{\mc{M}} \la e^{-\la t}u(\varphi_t(y))v(y)dt\, d\mu(y)+ \int_{T_\eps}^{\infty}\la e^{-\la t}\cjg u,1\cjd\cjg v,1\cjd dt \\
 & + \int_{T_\eps}^{\infty}\la e^{-\la t}C_t(u,v)dt.
\end{split}\]
The first term has norm bounded by $(1-e^{-\la T_\eps})\|u\|_{L^2}\|v\|_{L^2}$, the second term is equal to $e^{-\la T_\eps}\cjg u,1\cjd\cjg v,1\cjd$ and the last 
term has norm bounded by $\eps e^{-\la T_\eps}$. Therefore letting $\la\to 0$ we obtain
\[ \lim_{\la\to 0^+}(\la\cjg R_+(\la)u,v\cjd)= \cjg u,1\cjd\cjg v,1\cjd+\mc{O}(\eps)\]
and since $\eps$ is arbitrarily small we deduce that the residue of $R_+(\la)$ at $\la=0$ is the rank-1 operator $1\otimes 1$. The same argument shows 
that the residue of $R_-(\la)$ at $\la=0$ is $-(1\otimes 1)$ and for all $\la_0\in i\rr\setminus \{0\}$
\[ \lim_{\la\to \la_0}((\la-\la_0)\cjg R_\pm(\la)u,v\cjd)=0\]
where the limit is understood as a limit from the right half-plane ${\rm Re}(\la)>0$.

Conversely,  we can use the formula \eqref{Stone} and the meromorphy of $R_\pm(\la)$ to deduce that the $L^2$-spectrum of $iX$ is made of absolutely 
continuous spectrum and pure point spectrum.  Moreover if $0$ is the only pole on the imaginary line and if it is simple with residue $1\otimes 1$, 
then it means that the spectrum on $\{1\}^{\perp}$ is absolutely continuous and it is a classical fact that this implies that the flow is mixing 
(see \cite[Theorem VII.15]{ReSi}). 
%The spectral measure of $-iX$ on the absolutely continuous part is given by 
%\[dE(\la)=\frac{1}{2\pi} (R_+(i\la)-R_-(i\la))d\la= \frac{1}{2\pi} (R_+(i\la)+R_+(i\la)^*)d\la \] 
%which implies that it is analytic in $\la\in\rr$, and thus 
%\[\cjg e^{tX}u,v\cjd=\cjg u,1\cjd\cjg v,1\cjd + \frac{1}{2\pi}\int_{\rr}e^{it\la}\cjg (R_+(i\la)+R_+(i\la)^*)u,v\cjd d\la.\]
 %if $u,v\in C^\infty(\mc{M})$, we can use $-XdE(\la)=-i\la dE(\la)$ to write: 
%\[ |\cjg dE(\la)u,v\cjd|= \frac{1}{\la^{N}}|\cjg dE(\la)u,X^Nv\cjd|, \quad \forall  N\in \nn \]
%which implies that $C_t(u,v)=\mc{O}(t^{-N})$ for all $N\in\nn$.
\end{proof}

For Anosov geodesic flows, mixing was proved by Anosov \cite{An}, and this was extended to contact Anosov flows by Burns-Katok \cite{BuKa}; in that last case, the rate of mixing is $C_t(u,v)=\mc{O}(e^{-\eps |t|})$ for some $\eps>0$ if $u,v\in C^\infty(\mc{M})$ by Liverani \cite{Li}. 

\subsection{The operator $\Pi$}

When the Anosov flow is mixing then by Lemma \ref{mixing}, we know that the resolvents $R_\pm(\la)$ have a simple pole at $\la=0$ and using Laurent expansion at $\la=0$ of $R_\pm(\la)$ together with 
\eqref{adjoint}, we see that there exists an operator $R_0:\mc{H}^{s,r}\to \mc{H}^{s,r}$ such that 
\begin{equation}\label{laurentexp} 
R_+(\la)=\frac{1\otimes 1}{\la}+ R_0+\mc{O}(\la), \quad R_-(\la)=-\frac{1\otimes 1}{\la}- R_0^*+\mc{O}(\la) 
\end{equation}
and therefore as operators $C^\infty(\mc{M})\to C^{-\infty}(M)$
\begin{equation} \label{XR0}
-XR_0={\rm Id}-1\otimes 1=-R_0X, \quad X R_0^*={\rm Id}-1\otimes 1=R_0^*X.
\end{equation}
This identity extends to those Sobolev spaces on which the operators are bounded (as given by Theorem \ref{fauresjos}). 
In particular, by the Fredholm property of $X$ on $\mc{H}^{s,r}$ and $\mc{H}^{r,s}$, we deduce that 
the kernel of $X$ on $\mc{H}^{s,r}$ and $\mc{H}^{r,s}$ is simply made of constants.
The operator $R_0$ is simply obtained as the limit
\begin{equation}\label{defR0}
R_0=\lim_{\la\to 0}(R_+(\la)-\la^{-1}(1\otimes 1)).
\end{equation}
When the flow is not mixing, the same exact properties hold as long as ${\rm Res}_0R_+(\la)=1\otimes 1$, which is equivalent to say  
that ${\rm Res}_0R_+(\la)$ is self-adjoint and the flow is ergodic (so the projector $\Pi_{{\rm Eig}(0)}$ on the $L^2$-kernel is 
$1\otimes 1$).

We can now show 
\begin{theo}\label{thPi}
Let $\mc{M}$ compact and $X$ be a smooth vector field generating an Anosov flow preserving a smooth invariant probability measure $d\mu$. Assume that ${\rm Res}_0R_+(\la)=1\otimes 1$ (this is in particular true if the flow is mixing). 
For all $s>0$ and $r<0$, the operator $\Pi:=R_0+R_0^*:H^s(\mc{M})\to H^r(\mc{M})$ is bounded and 
satisfies 
\[\begin{gathered}
X\Pi f=0,\,\,\,  \forall f\in H^s(\mc{M}), \,\, \textrm{ and } \,\, \Pi Xf=0,\,\,\, \forall f\in H^{s+1}(\mc{M}),\\
f\in C^\infty(\mc{M})\Longrightarrow {\rm WF}(\Pi f)\subset E_u^*\cup E_s^*.
\end{gathered}\]
Let $f\in H^{s}(\mc{M})$ with $\cjg f,1\cjd=0$ and set $u_+:=-R_0f\in \mc{H}^{s,r}$, $u_-:=R_0^*f\in \mc{H}^{r,s}$ so that 
\[Xu_+=Xu_-=f\] 
by \eqref{XR0}. Then
$f\in \ker\Pi\cap H^s(\mc{M})$ if and only if there exists $s'>0$ and a solution 
$u\in H^{s'}(\mc{M})$ to $Xu=f$; 
in this case the solution $u$ is actually in $H^s(\mc{M})$ and 
is unique modulo constants, given by $u=u_+=u_-$. 
\end{theo}
\begin{proof}
The first part follows from the boundedness of the operator $R_0$ and $R_0^*$ in Theorem \ref{fauresjos} and the relations \eqref{XR0}. The wave front set description of $\Pi f$ is a consequence of Proposition \ref{dyatlovzworski} and \cite[Theorem 8.2.12]{Ho1}.
If $f\in H^s(\mc{M})$ is in $\ker \Pi$ then $u_+=u_-$ and $Xu_+=f$, moreover by Lemma \ref{regul} and taking 
$r=-\eps s/2$ for some $\eps>0$ small independent of $s$, one has $u_+\in H^{s\eps/2}(\mc{M})$. 
If there is another solution 
in $L^2(\mc{M})$ for some $\eps>0$, there is a flow invariant $L^2$ function and so it is constant. 
In fact, we can prove better regularity of $u_+$ using propagation of singularities, namely that 
\begin{equation}\label{v_+}
u_+\in H^s(\mc{M}).
\end{equation}
We know that $u_+\in \mc{H}^{s,r}\cap \mc{H}^{r,s}$ and thus, from the definition of the spaces $\mc{H}^{s,r}$ and $\mc{H}^{r,s}$, we have that 
$A_0u_+\in H^s(\mc{M})$ if $A_0\in \Psi^0(\mc{M})$ is microsupported in a sufficiently small 
conic neighborhood of $E_u^*\cup E_s^*$ and elliptic in a conic neighbourhood $W_0$ of 
$E_u^*\cup E_s^*$. 
By classical elliptic estimates, we also have that $A_1u_+\in H^{s}(\mc{M})$ if 
$A_1\in \Psi^0(\mc{M})$ has wave front set not intersecting the characteristic region 
$E_s^*\oplus E_u^*=\{(y,\xi)\in T^*\mc{M}; \xi(X(y))=0\}$, and we can assume $A_1$ elliptic outside a small conic
 neighborhood of $E^*_s\oplus E_u^*$, we call $W_1$ the region of ellipticity of $A_1$. 
 Moreover, for all $(y,\xi)\notin W_0\cup W_1$, the trajectory $\Phi_t(y,\xi)$ of the Hamilton flow of the principal symbol 
of $-iX$ (ie. the symplectic lift of $\varphi_t$) reach $W_1\cup W_0$ in either forward or backward finite time: 
this is a consequence of the fact that for all $(y,\xi)\notin E_s^*$, $\Phi_t(y,\xi)/|\Phi_t(y,\xi)|$ 
tend to $E_0^*\oplus E_u^*$ as $t\to +\infty$ and for all $(y,\xi)\notin E_u^*$, $\Phi_t(y,\xi)/|\Phi_t(y,\xi)|$ 
tend to $E_0^*\oplus E_s^*$ as $t\to -\infty$ (see \cite[Section 2]{FaSj} for example). We are going to apply  
the propagation of singularities for real principal type differential operators given in \cite[Proposition 2.5]{DyZw1}  that we recall now (see \cite[Section 3]{Ho0} for the original argument in the case of constant coefficients operators): assume that $u\in H^{-N}(\mc{M})$ for some $N>0$, and let $A,B,B_1\in \Psi^0(\mc{M})$ that are elliptic in respective conic subsets $\mc{A},\mc{B},\mc{B}_1 \subset T^*\mc{M}$ such that $Bu\in H^{s}(\mc{M})$, $B_1Xu\in H^{s}(\mc{M})$ for some $s>-N$; if for all $(y,\xi)$ in the microsupport (wave front set) of $A$ there is $T>0$ such that $\Phi_{-T}(y,\xi)\in \mc{B}$ and $\Phi_{t}(y,\xi)\in \mc{B}_1$ for all $t\in[-T,0]$, then $Au\in H^s(\mc{M})$.
Applied to our case, we thus deduce that  \eqref{v_+} holds since $-iXu_+\in H^s(\mc{M})$.

We now prove the converse. If there is $u\in H^{s'}(\mc{M})$ so that $Xu=f$ with $0<s'\leq s$, then 
$u-u_+$ is constant since $u\in \mc{H}^{s',r}$, $u_+\in \mc{H}^{s',r}$ by \eqref{nested2} 
and $-X$ is Fredholm on $\mc{H}^{s',r}$ with kernel given by constants. 
Similarly, $u-u_-$ is constant, and thus $u\in \mc{H}^{s,r}\cap \mc{H}^{r,s}$. Then $u_+-u_-=C$ for 
some $C\in \cc$ and 
$\Pi f=C$. Since $XR_0 (1)=XR_0^*(1)=0$ we have that $R_0(1)$ and $R_0^*(1)$ are constants   
and thus $\Pi(1)$ is constant. We obtain $\cjg \Pi f,1\cjd=\cjg f,\Pi(1)\cjd=0$ and thus $C=0$, showing $\Pi f=0$. We also have that 
$u_+\in H^s(\mc{M})$ by the arguments leading to \eqref{v_+}. This completes the proof.
\end{proof}

Taking $r=-s$ for $s>0$ fixed, $\Pi$ is self-adjoint as a map $H^{s}(\mc{M})\to H^{-s}(\mc{M})$ if we identify 
$H^{-s}(\mc{M})$ with the dual of $H^s(\mc{M})$, in the sense 
$\cjg \Pi f,f'\cjd=\cjg f,\Pi f'\cjd$ for all $f,f'\in H^s(\mc{M})$.
Moreover it maps any $H^s(\mc{M})$ to the space of invariant distributions defined by
\begin{equation}\label{defI}
\mc{I}:=\bigcap_{r<0} \mc{I}_r, \quad \mc{I}_r:=\{ w\in H^{r}(\mc{M}); \,\, Xw=0\}.
\end{equation}
We claim that the image of $\Pi$ is infinite dimensional for Anosov flows satisfying the assumptions of Theorem \ref{thPi}. 
Indeed, for each closed orbit $\gamma$ of $X$, there is a smooth function $f$  supported in an arbitrarily small tubular neighbourhood which equals $1$ on $\gamma$ and $\cjg f,1\cjd=0$, and $f\notin \ker \Pi$ since, if we had $\Pi f=0$, by Theorem \ref{thPi} this would imply that there exists $u\in C^\infty(\mc{M})$ such that $Xu=f$, and thus 
$\int_\gamma f=0$, contradicting the fact that $f=1$ near $\gamma$. 
By \cite[Theorem 3]{An}, there are countably infinitely many (disjoint) periodic orbits $\gamma_k$ and thus for any $N\in\nn$ we can construct smooth functions $(f_k)_{k=1,\dots,N}$ with disjoint supports, so that $\int_{\gamma_j}f_k=\delta_{jk}$ for $j,k\leq N$.
We deduce that $\dim {\rm span}\{\Pi f_k; k\leq N\}=N$ since if $\Pi(\sum_{k=1}^Na_kf_k)=0$ for some $a_k\in\cc$,   
 there exists $u\in C^\infty(\mc{M})$ such that $Xu=\sum_{k=1}^Na_kf_k$, thus 
  $0=\int_{\gamma_j}Xu=a_j$ for each $j\leq N$. Letting $N\to \infty$, we see that the range of $\Pi$ has infinite dimension.

The alternative expression \eqref{weaklim} for $\cjg \Pi f,\psi\cjd$ if $\cjg f,1\cjd=0$  is a direct consequence of 
\eqref{R_+R_-} and \eqref{laurentexp}.
We now give a direct corollary of Theorem \ref{thPi}, which proves the density of the range 
of $\Pi$ in the space of invariant distributions for the flow.
\begin{corr}\label{directcor}
Let $\mc{M}$ and $X$ satisfy the same assumptions as in Theorem \ref{thPi}. 
For each $s>0$, the space $\mc{I}_{-s}$ of invariant distributions in $H^{-s}(\mc{M})$ 
is the closure in $H^{-s}(\mc{M})$ of ${\rm Ran}(\Pi|_{H^s}):=\{\Pi f\in H^{-s}(\mc{M}); f\in H^s(\mc{M})\}$, which in turn is the orthogonal of $\ker (\Pi|_{H^s(\mc{M})})$. 
\end{corr}
\begin{proof}
If $f\in H^s(\mc{M})$ is orthogonal to $\mc{I}_{-s}$, then using that ${\rm Ran}(\Pi|_{H^s})\subset \mc{I}_{-s}$, we have $0=\cjg \Pi u,f\cjd=\cjg u,\Pi f\cjd$ 
for all $u\in H^s(\mc{M})$ thus $\Pi f=0$. Conversely, if $\Pi f=0$, by Theorem \ref{thPi} we get $Xu=f$ for some $u\in H^s(\mc{M})$, we want to show that $\cjg Xu,w\cjd=0$ for all $w\in \mc{I}_{-s}$: by \cite[Lemma E.41]{DyZw2} there is a sequence $u_j\in C^\infty(\mc{M})$
such that $u_j\to u$ in $H^{-s}(\mc{M})$ and $Xu_j\to Xu$ in $H^{-s}(\mc{M})$ as $j\to \infty$, thus 
$\cjg w,Xu\cjd=\lim_{j\to\infty} \cjg w,Xu_j\cjd=-\lim_{j\to\infty}\cjg Xw,u_j\cjd=0$. We deduce that the space $\mc{I}_{-s}^\perp:=\{ f\in H^s(\mc{M}); \forall w\in\mc{I}_{-s}, \, \cjg w,f\cjd=0\}$ is equal to $\ker \Pi|_{H^s(\mc{M})}$. On the other hand $({\rm Ran}(\Pi|_{H^s}))^\perp=(\bbar{{\rm Ran}(\Pi|_{H^s})})^\perp$ where the closure is in $H^{-s}(\mc{M})$ and $\ker \Pi|_{H^s}=({\rm Ran}(\Pi|_{H^s}))^\perp$ thus 
$\ker \Pi|_{H^s} =\mc{I}_{-s}^\perp\subset ({\rm Ran}(\Pi|_{H^s}))^\perp=\ker \Pi|_{H^s}$
and we deduce that $\bbar{{\rm Ran}(\Pi|_{H^s})}=\mc{I}_{-s}$.
\end{proof}

Using the operator $\Pi$, we also recover the smoothness result of \cite{DMM, Jo} for the solution of the cohomological equation and, indeed, we get a Sobolev version which does not seem to be available in the literature:
\begin{corr}\label{livsic}
With the same assumptions as in Theorem \ref{thPi}, if the flow is topologically transitive (which is the case for mixing flows), then 
for all $s>{\dim \mc{M}}/2$ and all $f\in H^{s}(\mc{M})$ satisfying  $\int_{\gamma}f=0$ 
for all closed orbits $\gamma$ of $X$, there exists
$u\in H^{s}(\mc{M})$ such that $Xu=f$. In particular if $f\in C^\infty(\mc{M})$ then $u\in C^\infty(\mc{M})$.
\end{corr}
\begin{proof} Let $n=\dim \mc{M}$.
If $f\in H^{n/2+\nu}(M)$ for $\nu\in (0,1)$, then $f\in C^{\nu}(\mc{M})$ and since the flow is assumed topologically transitive, 
by Livsic theorem for H\"older function \cite[Theorem 19.2.4]{KaHa}, we know that there exist $u\in C^\nu(\mc{M})$ so that $Xu=f$. 
Then we get $u\in H^{\nu'}(\mc{M})$ for all $\nu'<\nu$ since $\mc{M}$ is compact and so by Theorem \ref{thPi} we see that $f\in \ker \Pi$, 
which implies that there is $u\in H^{s}(\mc{M})$ so that $Xu=f$ and $u$ is unique modulo constants.
\end{proof}

\begin{rem}
The three results above hold as well if instead of assuming that the residue ${\rm Res}_0R_+(\la)$ is $1\otimes 1$,  we only assume that ${\rm Res}_0R_+(\la)$ is self-adjoint (which is equivalent to assuming that the residue is equal to the spectral projector on the $L^2$ kernel of $X$, by Lemma \ref{poles}), but then we need to take $f$ orthogonal to the $L^2$-kernel of $X$ and uniqueness modulo constants needs to be replaced by uniqueness modulo the $L^2$-kernel.
\end{rem}

\section{Anosov Geodesic flows}

In this section, we consider the special case of $\mc{M}=SM$ being the unit tangent bundle of a compact Riemannian manifold $(M,g)$ such that the geodesic flow $\varphi_t:SM\to SM$ of the metric $g$ is Anosov. 
We shall denote $\pi_0:SM\to M$ the natural projection $\pi_0(x,v)=x$ where $x\in M$ is the base point of the element 
$(x,v)\in SM$. 

\subsection{X-ray transform on functions} 

In this section, we study the operator $\Pi$ acting on pull-back of functions on $M$. Recall that the geodesic flow is mixing \cite{An}.
 The pull-back 
operator $\pi_0^*: C^\infty(M)\to C^\infty(SM)$ induces a push-forward map ${\pi_0}_*: C^{-\infty}(SM)\to C^{-\infty}(M)$ on distributions by 
\[ \cjg {\pi_0}_*u,\psi \cjd:=\cjg u,\pi_0^*\psi\cjd, \quad \forall \psi\in C^\infty(M).\]
We then define the operator 
\begin{equation}\label{Pi0} 
\Pi_0:= {\pi_0}_*\Pi\,  \pi_0^* :  \, C^\infty(M)\to C^{-\infty}(M).
\end{equation}
This operator corresponds exactly to the operator $I_0^*I_0$ which appears in 
the setting metric on manifolds with boundary, with $I_0$ being the X-ray transform on functions 
(see Section 5.1 in \cite{Gu} for explanations).
The goal of this section is to prove 
\begin{theo}\label{microlocalPi0}
If $(M,g)$ has Anosov geodesic flow, the operator $\Pi_0$ is an elliptic self-adjoint pseudo-differential operator of order $-1$, with principal 
symbol 
\[ \sigma(\Pi_0)(x,\xi)=C_n|\xi|_{g_x}^{-1}\]
where $C_n$ is a non-zero constant depending only on $n$. As a consequence, the kernel 
$\ker \Pi_0:=\{f\in C^{-\infty}(M);\,\, \Pi_0f=0\}$
is finite dimensional and its elements are smooth.
\end{theo}
\begin{proof}
Let us first show the first statement using Proposition \ref{dyatlovzworski}. We will first show that the wave front set of the Schwartz kernel of $\Pi_0$ is conormal to the diagonal, and then reduce the computation of the symbol to the case dealt with by Pestov-Uhlmann \cite{PeUh}.
We write if ${\rm Re}(\la)>0$ 
\[R_+(\la)=\int_0^\eps e^{-\la t}e^{tX}dt+e^{-\eps \la}e^{\eps X}R_+(\la).\]
where $\eps\geq 0$ is small and this extends meromorphically to $\cc$ when acting on smooth functions. At $\la=0$, 
we deduce from $\eqref{laurentexp}$ that the finite part of $R_+(\la)$ is
\[R_0=\int_0^\eps e^{tX}dt + e^{\eps X}R_0-\eps(1\otimes 1).\]
The last term is smoothing, we now describe the wave front set of the Schwartz kernel of ${\pi_0}_*e^{\eps X}R_0\pi_0^*$.  
By \cite[Theorem 8.2.4]{Ho1}, the wave front set of the Schwartz kernel of $e^{\eps X}$ is 
\[{\rm WF}(e^{\eps X})\subset \{(\varphi_{-\eps}(y),\eta,y,-d\varphi_{-\eps}(y)^T\eta); \,\, y\in SM, \eta\in 
T^*_{\varphi_{-\eps}(y)}(SM)\setminus\{0\}\}\]
and Proposition \ref{dyatlovzworski} gives the wave front set of $R_0$, thus by \cite[Theorem 8.2.14]{Ho1}, we deduce  
\[\begin{split}
{\rm WF}(e^{\eps X}R_0)\subset & \{(\varphi_t(y),(d\varphi_t(y)^{-1})^{T}\eta,y,-\eta);  \,\, t\leq -\eps,\,\, \eta(X(y))=0\}\cup (E_s^*\x E_u^*)\\
& \cup \{(\varphi_{-\eps}(y),\eta ,y,-d\varphi_{-\eps}(y)^T\eta);\,\, (y,\eta)\in T^*(SM)\setminus\{0\}\}
\end{split}\]
using that $E_s^*,E_u^*$ are invariant by the lifted flow $\Phi_t:T^*(SM)\to T^*(SM)$.
Then, we notice that the Schwartz kernel of ${\pi_0}_*e^{\eps X}R_0\pi_0^*$ is given by 
the push forward $(\pi_0\otimes \pi_0)_*K_\eps$ if $K_\eps$ is the kernel of $e^{\eps X}R_0$. Since by \cite[Proposition 11.3.3.]{FrJo} 
\[{\rm WF}((\pi_0\otimes \pi_0)_*K_\eps)\subset \{(\pi_0(y),\xi,\pi_0(y'),\xi'); \,\,  (y,d\pi_0(y)^T\xi,y',d\pi_0(y')^T\xi')\in {\rm WF}(K_\eps)\}\]
we deduce that ${\rm WF}({\pi_0}_*e^{\eps X}R_0\pi_0^*)\subset S_1 \cup S_2\cup S_3$ with 
\[\begin{split}
S_1:= \{ & (\pi_0(y),\xi,\pi_0(y'),\xi')\in T_0^*(M\x M); \,  (y,d\pi_0(y)^T\xi,y',d\pi_0(y')^T\xi')\in E_s^*\x E_u^*\}\\
S_2:=\{ & (\pi_0(\varphi_t(y)),\xi,\pi_0(y),\xi')\in T_0^*(M\x M); \,\, \exists \, t\leq -\eps,  \exists \, \eta, \eta(X(y))=0 , \\
            & \,\, d\pi_0(y)^T\xi'=-\eta, \,\, d\pi_0(\varphi_t(y))^T\xi=(d\varphi_t(y)^{-1})^{T}\eta\}\\
S_3:=\{( & \pi_0(\varphi_{-\eps}(y)),\xi,\pi_0(y),\xi')\in T_0^*(M\x M);\,\,  (d(\pi_0\circ \varphi_{-\eps})(y))^T\xi=-d\pi_0(y)^T\xi' \}
\end{split}\]
where $T_0^*(M\x M):=T^*(M\x M)\setminus \{0\}$.
Denote by $V=\ker d\pi_0\subset T(SM)$ the vertical bundle, and $H$ the horizontal bundle (cf. \cite[Chapter 1.3]{Pa}), these are orthogonal 
for the Sasaki metric $\cjg\cdot,\cdot\cjd_S$ and $X\in H$. Let $V^*,H^*\subset T^*(SM)$ defined by $H^*(V)=0$ and $V^*(H)=0$; $V^*$ is dual to $V$ and $H^*$ is dual to $H$ 
using this metric.
We have $E_s^*\cap H^*=\{0\}=E_u^*\cap H^*$ since $\rr X\oplus E_u\oplus V=\rr X\oplus E_s\oplus V=T(SM)$ (see for instance \cite[Theorem 2.50]{Pa}), therefore 
$S_1=\emptyset$. Now take a point $(x,\xi,x',\xi')\in S_2$, then write $x'=\pi_0(y)$ and $y=(x',v)$ with $v\in S_xM$, so 
$\varphi_t(x',v)=(x,v')$ for some $v'$, there exists $\eta=d\pi_0(y)^T\xi'\in H^*$ so that $(d\varphi_t(y)^{-1})^{T}\eta\in H^*$ 
for some $t\leq -\eps$, and $\eta(X)=0$. Taking the dual vector $\zeta\in T_y(SM)$ to $\eta\in T^*_y(SM)$ using the Sasaki 
metric, we have $\zeta\in H$ and $\cjg \zeta,X\cjd_S=0$. Now, let $J$ be the almost complex structure on $T(TM)$ so that $\cjg J\cdot,\cdot\cjd_S$ is the 
Liouville symplectic form on $TM$ (see \cite[Chapter 1.3.2]{Pa}), then $J$ maps $\zeta$ 
to $J\zeta\in V$. Since the flow preserves the symplectic form, one has $(d\varphi_t)^TJd\varphi_t=J$ where the transpose 
is with respect to the Sasaki metric, and then we see that $d\varphi_t(y) J\zeta=J(d\varphi_t(y)^{-1})^T\zeta\in V$ since 
$(d\varphi_t(y)^{-1})^{T}\zeta\in H$. We deduce that the points $x$ and $x'$ are conjugate points, which is not possible if the flow is Anosov, by a result of Klingenberg \cite[Theorem p.2]{Kl}. As a conclusion, $S_2=\emptyset$. Using finally the formula of $S_3$, we have shown that for $\eps>0$ smaller than the injectivity radius, ${\pi_0}_*e^{\eps X}R_0\pi_0^*$ has smooth Schwartz kernel except at $\Delta_\eps(M\x M):=\{(x,x')\in M\x M; d_g(x,x')=\eps\}$, where $d_g$ is the Riemannian distance.

We finally have to study the operator $L_\eps:=\int_{0}^\eps {\pi_0}_*e^{tX}\pi_0^*dt$, and we take $\eps$ smaller than the injectivity radius. This operator $L_\eps$ can be written as 
\begin{equation}\label{Leps} 
L_\eps f(x)=\int_{0}^\eps \int_{S_xM}f(\varphi_t(x,v))dv dt
\end{equation}
and it is a straightforward computation to check that its Schwartz kernel $L_\eps(x,x')$ is smooth outside  $\Delta(M\x M)
\cup \Delta_\eps(M\x M)$ 
with a conormal singularity of the form  $d_g(x,x')^{-n+1}$ at $\Delta(M\x M)$ if the dimension of $M$ is $n$. Since $\eps>0$ is arbitrary (in a small interval), this implies that ${\pi_0}_*R_0\pi_0^*$ has wave front set given by the conormal bundle $N^*\Delta(M\x M)$. In fact, the analysis of the singularity at $\Delta(M\x M)$ follows directly from Pestov-Uhlmann \cite[Lemma 3.1]{PeUh}: let $x_0\in M$
and multiply the kernel of ${\pi_0}_*R_0\pi_0^*$ with a smooth cut-off function $\psi$ which is $1$ near a point $(x_0,x_0)\in \Delta(M\x M)$ and supported in a 
neighborhood $\{(x,x')\in M\x M; d_g(x,x_0)+d_g(x',x_0)<\eps/2\}$, it is then equal, up to a smooth function, to the kernel $L_\eps(x,x')\psi(x,x')$. This distribution is the Schwartz kernel 
of a pseudo-differential operator of order $-1$ with principal symbol $C_n|\xi|_{g_x}^{-1}$: indeed from the formula  \eqref{Leps}, we see that  the Schwartz kernel of $2L_\eps$ coincides near $(x_0,x_0)$ 
with the Schwartz kernel of the operator $I_0^*I_0$ considered in \cite{PeUh} where $I_0$ is the X-ray transform on functions on a geodesic ball of center $x_0$ and radius $\eps$ (which is a simple domain); the detailed computation of the symbol at $x_0$ is thus exactly the same as in \cite[Lemma 3.1]{PeUh}. To conclude the proof of the structure of $\Pi_0$, we argue that the same exact argument applies for ${\pi_0}_*R_0^*\pi_0^*$ (in fact this is just the adjoint ${\pi_0}_*R_0\pi_0^*$ and its Schwartz kernel has the exact same property as ${\pi_0}_*R_0\pi_0^*$). The statement about $\ker \Pi_0$ is a direct consequence of ellipticity.
\end{proof}

\begin{rem} In constant negative curvature, we can use representation theory to give an expression of 
the operator $\Pi_0$ in term of the Laplacian: it turns out to be an explicit function of the Laplacian on the manifold. We refer to \cite[Appendix]{GuMo} for the computation.
\end{rem}

We remark that  for $f\in C^{-\infty}(M)$, then by \cite[Theorem 8.2.4]{Ho1}) 
\begin{equation}\label{WFpi0f}
{\rm WF}(\pi_0^*f)\subset \{(y,d\pi_0(y)^T\eta); \,\, (\pi_0(y),\eta)\in {\rm WF}(f)\}\subset H^*
\end{equation}
and so, using Theorem 8.2.13 in \cite{Ho1} and the fact that $H^*\cap E_u^*=0=H^*\cap E_s^*$ as in the proof of Theorem \ref{microlocalPi0}, the operator $R_0\pi_0^*$ and $R_0^*\pi_0^*$ acts on $C^{-\infty}(M)$ continuously, so that 
\[ \Pi\, \pi_0^* : C^{-\infty}(M)\to C^{-\infty}(SM)\] 
is continuous.  In fact, we can say more: 
\begin{lemm}\label{boundednegative}
1) For all $s>0$, the following operator is bounded
\[\Pi\,  \pi_0^* : H^{-s}(M)\to H^{-s}(SM).\]
and  the kernel of $\Pi \, \pi_0^*: C^{-\infty}(M)\to C^{-\infty}(SM)$ is trivial.\\ 
2) Assume that there exists $u\in H^{s}(SM)$ with $Xu=\pi_0^*f$ for some 
$f\in H^{s-1}(M)$ satisfying $\cjg f,1\cjd =0$ and with $s\in(0,1)$, then $u$ is constant.
\end{lemm} 
\begin{proof} For $f\in H^{-s}(M)$ with $s>0$, we have 
$\pi_0^*f\in H^{-s}(SM)$ and, by \eqref{WFpi0f}, $B\pi_0^*f\in C^\infty(SM)$ for any pseudo-differential operator 
$B\in \Psi^0(M)$ microsupported in any open conic neighborhood of $E_s^*\cup E_u^*$ not intersecting $H^*$. Using \eqref{fBf2}, $|\cjg(1-B)\pi_0^*f,f'\cjd| \leq C||f'||_{\mc{H}^{s/\eps,r}}$ 
and  thus $\pi_0^*f\in (\mc{H}^{s/\eps,r})^*$ for any $r<0$. Similarly using \eqref{fBf}, 
$\pi_0^*f\in (\mc{H}^{r,s/\eps})^*$ for any $r<0$ and some $\eps>0$ small independent of $s$ (using the notation of Theorem \ref{fauresjos}). Therefore by \eqref{dual}, we obtain 
boundedness of $R_0\pi_0^*:H^{-s}(M)\to (\mc{H}^{r,s/\eps})^*$ and $R_0^*\pi_0^*:H^{-s}(M)\to (\mc{H}^{s/\eps,r})^*$. To show that in fact they map to $H^{-s}(SM)$, we will use propagation of singularities. First, like in \eqref{fBf}, we notice that for any $r<0$, $AR_0\pi_0^*: H^{-s}(M)\to H^{-r}(SM)$ is bounded 
if $A\in \Psi^0(SM)$ is microsupported in a small conic neighborhood of $E_u^*$. Then, by ellipticity and propagation of singularities \cite[Propositions 2.4 and 2.5]{DyZw1} as in the proof of Theorem \ref{thPi}, we obtain that $BR_0\pi_0^*: H^{-s}(M)\to H^{-s}(SM)$ is bounded for any 
$B\in \Psi^0(SM)$ whose microsupport does not intersect $E_s^*$. To conclude the argument, we use the propagation estimate with radial sink\footnote{It can be checked that, since $X^*=-X$, 
the $m_0>0$  parameter in Proposition 2.7 of \cite{DyZw1} can be taken as small as we like.} 
\cite[Proposition 2.7]{DyZw1}: for any $s>0,N\geq s$ 
and any $B_1\in \Psi^0(SM)$ elliptic near $E_s^*$, 
there exist $A\in \Psi^0(SM)$ elliptic near $E_s^*$, $B\in \Psi^0(SM)$ 
with ${\rm WF}(B)\cap E_s^*=\emptyset$ and microsupported in the region where $B_1$ is elliptic 
and $C>0$ such that for any $u\in C^\infty(SM)$,
\[ ||Au||_{H^{-s}(SM)}\leq C(||Bu||_{H^{-s}(SM)}+||B_1Xu||_{H^{-s}(SM)} +||u||_{H^{-N}(SM)}).\]
Applying this with $u=R_0\pi_0^*f$ for $f\in H^{-s}(M)$, we deduce that $AR_0\pi_0^*:H^{-s}(M)\to H^{-s}(SM)$
is bounded for some $A\in \Psi^0(SM)$ elliptic near $E_s^*$, and therefore $R_0\pi_0^*:H^{-s}(M)\to H^{-s}(SM)$ is bounded. The same argument works with $R_0^*\pi_0^*$ and we obtain the boundedness statement in 1). 
To prove that the kernel is trivial, we remark that, by Theorem \ref{microlocalPi0}, if $\Pi \,\pi_0^*f=0$ then $f$ is smooth, and by Theorem \ref{thPi} there exists $u\in C^\infty(SM)$ so that $Xu=\pi_0^*f$.  By 
the result of Dairbekov-Sharafutdinov \cite{DaSh} one has $f=0$ (this follows directly from the so-called \emph{Pestov identity} - see \cite[Proposition 2.2.]{PSU2}).

Let $u\in H^{s}(SM)$ with $Xu=\pi_0^*f$ for some $s>0$ and $f\in H^{s-1}(M)$. Recall from Theorem \ref{fauresjos} that $X$ is 
Fredholm on $(\mc{H}^{s',r})^*$ and $(\mc{H}^{r,s'})^*$ for any $r<0$ and $s'>0$, with kernel the constants and respective inverse operator $R_0$ and $R_0^*$.
So using the fact that $\pi^*_0f\in (\mc{H}^{(1-s)/\eps,r})^*\cap (\mc{H}^{r,(1-s)/\eps})^*$ for some small $\eps>0$ and all $r<0$, and 
$u\in H^{s}(SM)\subset (\mc{H}^{(1-s)/\eps,r})^*\cap (\mc{H}^{r,(1-s)/\eps})^*$ if $|r|\leq s$, 
then we deduce by the Fredholm property of $X$ that $u=-R_0\pi_0^*f=R_0^*\pi_0^*f$. As a consequence, we have that $f\in \ker \Pi \,\pi_0^*$, thus $Xu=f=0$ and $u$ is constant since $u\in L^2$. 
\end{proof}
It is not clear if injectivity holds for $\Pi_0$, at least we do not see why $\Pi_0f=0$ would imply
$\Pi \pi_0^*f=0$.  Another corollary of Theorem 
\ref{microlocalPi0} is the existence of invariant distributions with prescribed push-forward on $M$.
\begin{corr}\label{surj0}
Let  $s\in\rr$ and $r<0$, then there exists $C>0$ such that for each $f\in H^{s}(M)$, there exists $w\in C^{-\infty}(SM)$ so that $Xw=0$ and ${\pi_0}_*w=f$ and 
\[\begin{gathered} 
\|w\|_{H^{s-1}(SM)}\leq C\|f\|_{H^{s}(M)} \quad \textrm{ if }s<1, \\  
\|w\|_{H^r(SM)}\leq C\|f\|_{H^{1}(M)} \quad \textrm{ if }s\geq 1. 
\end{gathered}\]
If $s>1$, for any $A\in \Psi^0(SM)$ with wave front set not intersecting $E_u^*\cup E_s^*$, there is $C>0$ such that for each $f\in H^{s}(M)$, the invariant distribution $w$ satisfies $||Aw||_{H^{s-1}(SM)}\leq C||f||_{H^s(M)}$.
\end{corr}
\begin{proof} If $f$ is constant the result is obvious, so we can assume that $\cjg f,1\cjd=0$.
 For $s\not=0$, define the sesquilinear product $B_s$ on $C^\infty(M)$
\[ B_s(u,u'):= \cjg \Pi\pi_0^*u,\Pi\pi_0^*u'\cjd_{H^{-|s|}(SM)}+\cjg \Pi_0u,\Pi_0u'\cjd_{H^{1-s}(M)}.\] 
Using the boundedness 1) in Lemma \ref{boundednegative}, and the fact that $\Pi_0$ is an elliptic pseudo-differential operator of order $-1$, we have that there exists $C>0$ and $K:H^{-s}(M)\to H^{-s}(M)$ a compact operator such that for all $u\in H^{-s}(M)$
\[ B_s(u,u)\geq \|\Pi_0u\|^2_{H^{1-s}(M)}\geq C\|u\|^2_{H^{-s}(M)}-\|Ku\|^2_{H^{-s}(M)}.\]
Since $\Pi \, \pi_0^*$ is injective by Lemma \ref{boundednegative}, it is easy to see from the compactness of $K$ that there is $C'>0$ such that 
\[ B_s(u,u)\geq C'\|u\|^2_{H^{-s}(M)},\]
and thus the completion of $C^\infty(M)$ for the product $B_s$ is $H^{-s}(M)$.
Using Riesz representation theorem, for all $f\in H^{s}(M)$ there exists $u\in H^{-s}(M)$ such that 
$B_s(u,u')=\cjg f,u'\cjd_{L^2}$ for all $u'\in H^{-s}(M)$, with $||u||_{B_s}\leq C||f||_{H^s(M)}$ for some $C$ independent of $f$. We use the norm $||u||_{H^{s}(N)}:=||\Lambda_{N}^{s}u||_{L^2(N)}$ if $N=M$ or $N=SM$ and $\Lambda_N\in \Psi^{1}(N)$ is a fixed positive elliptic operator on $N$. 
This implies that  $w:=\Pi \tilde{u}$ with $\tilde{u}:=(\Lambda_{SM}^{-2|s|}\Pi \pi_0^*u+\pi_0^*\Lambda_M^{-2s+2}\Pi_0u)$
satisfies $Xw=0$ and ${\pi_0}_*w=f$ and $w\in H^{s-1}(SM)$ if $s-1<0$ while $w\in H^{r}(SM)$ if $s\geq 1$: 
indeed $\Lambda_{SM}^{-2|s|}\Pi \pi_0^*u\in H^{|s|}(SM)$ and $\pi_0^*\Lambda_M^{-2s+2}\Pi_0u\in H^{s-1}(SM)$
thus the regularity of $w$ follows from Lemma \ref{boundednegative} and Theorem \ref{thPi}.

The fact that $||Aw||_{H^{s-1}(SM)}\leq C||f||_{H^{s}(M)}$ if $A\in \Psi^0(SM)$ has microsupport not 
intersecting $E_s^*\cup E_u^*$ follows from an argument as in the proof of Lemma \ref{boundednegative}:
if $s>1$, the boundedness of $R_0$ on $\mc{H}^{s-1,r}$ for any $r<0$ implies that
$||BR_0 \tilde{u}||_{H^{s-1}(SM)}\leq C||f||_{H^s(M)}$ 
if $B\in \Psi^0(SM)$ is microsupported in a small enough conic neighborhood of 
$E_u^*$, then by ellipticity and propagation of singularities, for any $A$ with microsupport not intersecting $E_s^*$ there is $C>0$ so that $||AR_0 \tilde{u}||_{H^{s-1}(SM)}\leq C||f||_{H^s(M)}$, and then the same argument applies with 
$AR_0^*\tilde{u}$ by exchanging $E_s^*$ with $E_u^*$, this gives the desired result.
\end{proof}
This statement gives a more precise result than that of Paternain-Salo-Uhlmann \cite[Theorem 1.2]{PSU2} when $f$ has some regularity (using Pestov identity and Fourier decomposition \`a la Guillemin-Kazhdan, they obtain the same existence result but only for $s=0$).
 
\subsection{X-ray transform on symmetric tensors}
Consider the space of symmetric $m$-cotensors $C^\infty(M, \otimes_S^mT^*M)$. Then there is a natural map 
\[ \pi_m^*: C^\infty(M, \otimes_S^mT^*M)\to C^{\infty}(SM), \quad (\pi_m^*f)(x,v)=\cjg f(x),\otimes^m v\cjd.\]
The vertical Laplacian $\Delta_v: C^\infty(SM)\to C^\infty(SM)$ can be defined using the Riemannian metric on each fiber $S_xM$, and its spectral decomposition induces an isomorphism  
\[ L^2(SM)=\bigoplus_{m=0}^\infty H_m\]
where $H_m$ are $L^2$ sections of a smooth vector bundle over $M$ corresponding to the decomposition of a function into spherical harmonics of degree $m$ in the fibers $S_xM\simeq S^{n-1}$. Notice that spherical harmonics of degree $m$ correspond to restrictions of harmonic homogeneous polynomials on $\rr^n$ and $H_m$ identifies via $\pi_m^*$ to the space of $L^2$ sections of the bundle 
\[ E_m:=\{ q \in \otimes_S^mT^*M; \mc{T}(q)=0\}\]
where $\mc{T}: \otimes_S^mT^*M\to  \otimes_S^{m-2}T^*M$ is the trace defined by contracting with the Riemannian metric:
\begin{equation}\label{deftrace}  
\mc{T}(q)(v_1,\dots,v_{m-2}):=\sum_{i=1}^{n}q(e_i, e_i,v_1,\dots,v_{m-2})
\end{equation}
if $(e_1,\dots,e_n)$ is an  orthonormal basis of $TM$. Consider the operator 
$D:C^\infty(M, \otimes_S^mT^*M)\to C^\infty(M, \otimes_S^{m+1}T^*M)$ defined by $D=\mc{S}\circ \nabla$ where $\nabla$ is the Levi-Civita covariant derivative and $\mc{S}$ is the orthogonal projection on symmetric tensors. Each section of $\otimes_S^mT^*M$ can be decomposed as a sum of sections of
$E_j$ for $j\leq m$. The adjoint of $D$ is given by $D^*=-\mc{T}\circ D$ and is called the divergence. Then the flow $X$ acting 
on smooth sections of $E_m$, viewed as elements of $C^\infty(SM)$ through $\pi_m^*$, satisfies
\[ X : C^\infty(M,E_m)\to C^\infty(M,E_{m-1}\oplus E_{m+1}),\] 
where it decomposed as $X=X_++X_-$ with $X_\pm:C^\infty(M,E_m)\to C^\infty(M,E_{m\pm 1})$
and $X_+=D$ while $X_-=-X_+^*=-\tfrac{m}{n+2m-2}D^*$.
We refer to \cite{GK2} and \cite[Section 3]{PSU2} for further details and dicussions about this decomposition.

The map $\pi_m^*: C^\infty(M,\otimes_S^mT^*M)\to C^\infty(SM)$ induces a push-forward on distributions 
\[ {\pi_m}_*: C^{-\infty}(SM)\to C^{-\infty}(M,\otimes_S^mT^*M), \quad \cjg {\pi_m}_*u,\psi\cjd:= \cjg u,\pi_m^*\psi\cjd\]
where the pairing uses the metric $g$.
We now show
\begin{theo}\label{microlocalPim}
The operator $\Pi_m:= {\pi_m}_*\Pi\, \pi_m^*$ is a self-adjoint pseudo-differential operator of order $-1$ on the bundle 
$\otimes_S^mT^*M$, which is elliptic on $\ker D^*$ in the sense that there exist pseudo-differential operators $P,S,R$ with respective 
order $1,-2,-\infty$ so that 
\begin{equation}\label{parametrix}
P\Pi_m={\rm Id}+DSD^*+R.
\end{equation}
\end{theo}
\begin{proof} We follow the proof to Theorem \ref{microlocalPi0}. Take two points $x_0,x_0'$ in $M$, then we want to analyze the Schwartz kernel of $\Pi_m$ near $(x_0,x_0')\in M\x M$. Take two cutoff functions $\chi,\chi'$ 
supported in small neighborhood $V_{x_0}$ of $x_0$ and $V_{x_0'}$ of $x_0'$
so that the bundle $\otimes_S^mT^*M$ has a smooth orthonormal basis $(e_1(x),\dots,e_{N(m)}(x))$ on $V_{x_0}$
and $(e'_1(x),\dots,e'_{N(m)}(x))$ on $V_{x'_0}$ with $N(m)=\rank \otimes_S^mT^*M$. 
A smooth section $\psi$ of $\otimes_S^mT^*M$ can be written near $x_0$
\[\psi(x)=\sum_{j=1}^{N(m)}\cjg \psi(x),e_j(x)\cjd_{g}e_j(x) \]
and a similar decomposition for $\psi'$ supported near $x_0'$.
The Schwartz kernel $K_m$ of $\Pi_m$ near $(x_0,x_0')\in M\x M$ can be analyzed by considering 
$\chi \Pi_m \chi'$, which in turn is given by: for all $\psi,\psi'\in C^\infty(M,\otimes_S^mT^*M)$ 
\[ \cjg \chi \Pi_m \chi' \psi',\psi\cjd=\sum_{j,j'} \Big\cjg \Pi \, e'_{j'}\chi' \pi_0^*(\cjg \psi',e'_{j'}\cjd_{g}),
e_{j}\chi \pi_0^*(\cjg \psi,e_j\cjd_{g})\Big\cjd.
\]
The Schwartz kernel $\chi(x)\chi'(x) K_m(x,x')$ of $\chi \Pi_m \chi'$ can be viewed as a matrix valued distribution on $V_{x_0}\x V_{x_0'}$ using the local bases $(e_j)_j$ and $(e'_j)_j$ and its $(j,j')$ component is given by 
\[ K_m^{j,j'}=(\pi_0\otimes \pi_0)_* (\chi(x)\chi'(x') \cjg K(x,x')e'_{j'}(x'),e_j(x)\cjd)\]
where $K$ is the Schwartz kernel of $\Pi$. Since multiplying by a smooth function does 
not make the wave front set larger, we are reduced to the exact same analysis we did in the proof of Theorem \ref{microlocalPi0}.
Then we deduce that the operator $\Pi_m$ has smooth kernel outside the diagonal $\Delta(M\x M)$, the wave front set is contained in the conormal bundle to the diagonal and $\Pi_m$
can be written as 
\[\Pi_m = \int_{-\eps}^\eps {\pi_m}_*e^{tX}\pi_m^*dt+{\pi_m}_*e^{\eps X}(R_0+R_0^*)\pi_m^*+\textrm{ smoothing}\]  
for small $\eps$ with ${\pi_m}_*e^{\eps X}(R_0+R_0^*)\pi_m^*$ having a smooth Schwartz kernel outside the set 
$\{d_g(x,x')=\eps\}$. Then, just as in the proof of Theorem \ref{microlocalPi0}, we are reduced to analyze the integral kernel of
\[ \int_{0}^\eps {\pi_m}_*e^{tX}\pi_m^*dt\] 
close to $\Delta(M\x M)$. It follows from Sharafutdinov-Skokan-Uhlmann \cite[Theorem 2.1]{SSU} that, 
after multiplying by a smooth cutoff function equal to $1$ near $\Delta(M\x M)$ and supported in $\{d_g(x,x')<\eps/2\}$, this is 
a pseudo-differential operator of order $-1$ if $\eps>0$ is chosen smaller than the radius of injectivity. It is also shown in \cite[Theorem 3.1]{SSU} that there exists pseudo-differential operator 
$P,S,R$ as announced above. We notice that, even though \cite{SSU} deal with the case of simple manifolds, 
all their computations are local, and near a point $(x_0,x_0)\in \Delta(M\x M)$, the operator $\int_{0}^\eps {\pi_m}_*e^{tX}\pi_m^*dt$ has the same conormal singularity as the operator $I^*I$ acting on symmetric $m$-tensors on a small disk centered at $x_0$ (which is the case considered in \cite{SSU}).
\end{proof}

Just as for $m=0$, using the wave front set of $\Pi$, we see that 
\[ \Pi\, {\pi_m}_*: C^{-\infty}(M,\otimes_S^mT^*M)\to C^{-\infty}(SM) \]
is well-defined. The operator $D$ is elliptic on sections of $\otimes_S^mT^*M$ and thus the range 
of $D:H^{s}(M,\otimes_S^mT^*M)\to H^{s-1}(M,\otimes_S^{m+1}T^*M)$ 
is closed for any $s\in\rr$, and we have $\ker D^*={\rm Ran}(D)^\perp$ where ${\rm Ran}(D)$ is the range of $D$. We set the Hilbert space norm  
\[\|f\|_{H^s(M,\otimes_S^mT^*M)}:=\|(1+D^*D)^{s/2}f\|_{L^2(M,\otimes_S^mT^*M)}.\]
Then  $H^{s}(M,\otimes_S^mT^*M)\cap \ker D^*$ is a Hilbert space with this norm, and by Riesz representation theorem, we can describe the dual with respect to $L^2$ (ie. distributional) pairing as  
\[(H^{s}(M,\otimes_S^mT^*M)\cap \ker D^*)^*=H^{-s}(M,\otimes_S^mT^*M)\cap \ker D^*(1+D^*D)^{-s}.\]
We then get
\begin{lemm}\label{boundednegativem}
1) There exists $\eps>0$ so that for all $s>0$, the following operator is bounded
\[\Pi\,  \pi_m^* : H^{-s}(M,\otimes_S^mT^*M)\to H^{-s}(SM).\]
and the kernel of $\Pi \, \pi_m^*: C^{-\infty}(M,\otimes_S^mT^*M)\cap \ker D^*\to C^{-\infty}(SM)$ is finite dimensional.\\
2) The kernel of $\Pi \, \pi_m^*$ on $\ker D^*$ consists of those $f\in C^\infty(M,\otimes_S^mT^*M)\cap \ker D^*$ 
such that there exists $u\in C^{\infty}(SM)$ with $Xu=\pi_m^*f$.\\ 
3) Assume that there is $u\in H^{s}(SM)$  with $Xu=\pi_m^*f$ for some 
$f\in H^{s-1}(M,\otimes_S^mT^*M)\cap \ker D^*$ with $s\in(0,1)$, then $u\in C^\infty(SM)$.
\end{lemm} 
\begin{proof} The proof of the first boundedness result is exactly the same as for the case $m=0$. 
To prove that the kernel is finite dimensional, we remark that by Theorem \ref{microlocalPim}, 
if $\Pi \,\pi_m^*f=0$ and $D^*f=0$, then $\Pi_m f=0$ and so $f$ is smooth by applying \eqref{parametrix} to $f$. Then 
by Theorem \ref{thPi} there exists $u\in C^\infty(SM)$ so that $Xu=\pi_m^*f$.  Conversely, elements $f$ so that 
$Xu=\pi_m^*f$ for some $u$ smooth satisfy $\pi_m^*f\in \ker \Pi$ by Theorem \ref{thPi} (in particular, note that  $f$ integrates to $0$ along closed geodesics). 
The proof of the last statement is the same as for $m=0$ thus we do not repeat it. 
\end{proof}

\begin{rem} This gives an alternative (microlocal) proof of the result of Dairbekov-Sharafutdinov \cite[Theorem 1.5.]{DaSh} on the finite dimensionality of the kernel of the X-ray transform on $m$-cotensors on Anosov manifolds. 
By results of Croke-Sharafutdinov \cite{CS} and Lemma \ref{boundednegativem}, we deduce that the kernel of $\Pi\, \pi_{m}^*$ on $\ker D^*$ is trivial if $(M,g)$ is Anosov with non-positive curvature, and is always trivial when $m=1$ for Anosov manifolds by \cite[Theorem 1.3]{DaSh}.
\end{rem}

Finally we get existence of invariant distributions with prescribed push-forward ${\pi_m}_*$. 
\begin{corr}\label{corsurjm}
Let $m\geq 1$,  $s\in\rr$ and $r<0$, then there exists $C>0$ such that for each $f\in H^{s}(M,\otimes_S^mT^*M)\cap \ker D^*$ with $\cjg f,k\cjd_{L^2}=0$ for all $k\in \ker \Pi\pi_m^*\cap \ker D^*$, 
there exists $w\in C^{-\infty}(SM)$ so that $Xw=0$, ${\pi_m}_*w=f$ and
\[\begin{gathered} 
\|w\|_{H^{s-1}(SM)}\leq C\|f\|_{H^{s}(M)} \quad \textrm{ if }s<1, \\  
\|w\|_{H^r(SM)}\leq C\|f\|_{H^{1}(M)} \quad \textrm{ if }s\geq 1. 
\end{gathered}\]
If $s>1$, for any $A\in \Psi^0(SM)$ with wave front set not intersecting $E_u^*\cup E_s^*$, there is $C>0$ such that  for each $f$ as above, the invariant distribution $w$ satisfies $||Aw||_{H^{s-1}(SM)}\leq C||f||_{H^s(M)}$.
\end{corr}
\begin{proof} The proof is essentially the same as Corollary \ref{surj0}, thus we just focus on the differences.
Take $s\not=0$ and define the sesquilinear form on $C^\infty(SM)$
\[ B_s(u,u'):= \cjg \Pi \pi_m^*u,\Pi\pi_m^*u'\cjd_{H^{-|s|}(SM)}+\cjg \Pi_m u,\Pi_mu'\cjd_{H^{1-s}(M,\otimes_S^mT^*M)}.\] 
The $H^s$-norm on $\otimes_S^mT^*M$ is defined by $||u||_{H^s}=||\Lambda^su||_{L^2}$ where 
$\Lambda:=(1+D^*D)^{1/2}$. Then we conjugate \eqref{parametrix} by $\Lambda^{2s}$ on the left so that
\[  \Lambda^{2s}P\Pi_m\Lambda^{-2s}={\rm Id}+\Lambda^{2s}DSD^*\Lambda^{-2s}+\tilde{R}\]
where $\tilde{R}$ is smoothing. The left-hand side can be written as $P\Pi_m+Q$ with $Q\in \Psi^{-1}(SM)$, thus there exists $C>0$ and $T:H^{-s}(M,\otimes_S^mT^*M)\to H^{-s}(M,\otimes_S^mT^*M)$ compact such that for all 
$u\in H^{-s}(M,\otimes_S^mT^*M)\cap \ker D^*\Lambda^{-2s}$
\[ B_s(u,u)\geq \|\Pi_m u\|^2_{H^{1-s}(M,\otimes_S^mT^*M)}\geq C\|u\|^2_{H^{-s}(M,\otimes_S^mT^*M)}-
\|Tu\|^2_{H^{-s}(M,\otimes_S^mT^*M)}.\]
By Lemma \ref{boundednegativem}, the kernel $K_m$ of $\Pi \pi_m^*$ in $C^{-\infty}(M,\otimes_S^mT^*M)\cap \ker D^*$ is a finite dimensional space included in $C^\infty(M,\otimes_S^mT^*M)$, and we claim that 
if $u\in \ker \Pi\pi_m^*\cap \ker D^*\Lambda^{-2s}$ is in $H^{-s}(M,\otimes_S^mT^*M)$ and $\cjg u,k\cjd_{L^2}=0$ for all $k\in K_m$, then $u=0$: indeed we can decompose $u=u_0+Du'$ with $D^*u_0=0$ and $u'\in H^{-s+1}
(M,\otimes_S^{m-1}T^*M)$ and we get $\Pi\pi_m^*u_0=0$ since 
$\Pi\pi_m^*Du'=\Pi X\pi_{m-1}u'=0$; then $u_0\in K_m$ and $0=\cjg u,k\cjd_{L^2}=\cjg u_0+Du',k\cjd_{L^2}=\cjg u_0,k\cjd_{L^2}$ for each $k\in K_m$ and we get $u_0=0$. Since $D^*\Lambda^{-2s}u=0$, we have 
$D^*\Lambda^{-2s}Du'=0$, which gives $Du'=0$ (by pairing with $u'$), and thus $u=0$.
We conclude that there is $C>0$ such that 
\[B_s(u,u)\geq C\|u\|^2_{H^{-s}(M,\otimes_S^mT^*M)}\]
for each $u\in \mc{B}_s:=H^{-s}(M,\otimes_S^mT^*M)\cap \ker D^*\Lambda^{-2s}\cap (\Lambda^{2s}K_m)^\perp$ (the orthogonal is with respect to $H^{-s}$ product).
Then $\mc{B}_s$ is a Hilbert space with the $B_s$ scalar product.
Using Riesz representation theorem, for all $f\in H^{s}(M,\otimes_S^mT^*M)\cap \ker D^*$ satisfying 
$\cjg f,k\cjd_{L^2}=0$ for all $k\in K_m$, there exists $u\in \mc{B}_s$ such that 
$B_s(u,u')=\cjg f,u'\cjd_{L^2}$ for all $u'\in \mc{B}_s$, with $||u||_{B_s}\leq C||f||_{H^s}$ for some $C$ independent of $f$. From this, there is $w:=\Pi\Lambda^{-2|s|}\Pi\pi_m^*u+\Pi \pi_m^*\Lambda^{2-2s}\Pi_mu$ such that $Xw=0$, ${\pi_m}_*(w)=f+k+\Lambda^{-2s}Dq$ for some $q\in H^{s+1}(M;\otimes_S^{m-1}T^*M)$ and $k\in K_m$. We apply $D^*$ to this identity and since $D^*{\pi_m}_*\Pi=0$, we get $D^*\Lambda^{-2s}Dq=0$ and thus $Dq=0$. Now we use that 
$\cjg f,k\cjd_{L^2}=0$ and $\cjg {\pi_m}_*w,k\cjd_{L^2}=0$ to deduce that $k=0$. The regularity of $w$ is just as in the case $m=0$.
\end{proof}

\begin{rem} An analysis similar to what is done in Theorem \ref{microlocalPi0} and \ref{microlocalPim} shows that the operators 
${\pi_\ell}_*\Pi \, \pi_m^*$ are pseudo-differential of order $-1$ for all $m,\ell$. In particular, 
this implies that the components $w_\ell:={\pi_\ell}_*w$ of $w$ in Corollaries \ref{surj0} and \ref{corsurjm} are $H^{s-1}(M,E_\ell)$ if $f\in H^{s}(M,\otimes_S^m T^*M)$. In particular if $f\in C^\infty(M,\otimes_S^m T^*M)$ then the distribution $v$ in Corollaries \ref{surj0} and \ref{corsurjm} has smooth coefficients 
in the vertical Fourier decomposition (matching with the work \cite{PSU2} for $m=0$). This also follows from the 
wave front set property of $w$ and applying ${\pi_\ell}_*$.
\end{rem}

\subsection{Injectivity of X-ray transform on tensors for Anosov surfaces}

Consider an oriented compact Riemannian surface $(M,g)$ with Anosov geodesic flow $\varphi_t:SM\to SM$. As before,  let $X$ be the smooth vector field generating the flow. The manifold $SM$ is a circle bundle over $M$, equipped with 
a natural action 
 \[S^1 \x SM\to SM, \quad e^{i\theta}.(x,v)=(x,R_\theta(v))\]
where $R_\theta$ is the rotation of an angle $+\theta$ in the fiber. The action is generated by a vector field 
$V$ defined by $Vf(x,v)=\pl_tf(e^{it}.(x,v))|_{t=0}$. We let $X_\perp:=[X,V]$, and it can be checked 
that $(X,X_\perp,V)$ form a basis of $T(SM)$, which is orthonormal for the Sasaki metric. 
The tangent space $T(SM)$ splits as $T(SM)=\mc{V}\oplus \mc{H}$
where $\mc{V}=\rr V=\ker d\pi_0$ is the vertical space and $\mc{H}$ is the horizontal space defined using Levi-Civita connection (cf \cite{Pa}). In particular, one has $\mc{H}={\rm span}(X,X_\perp)$.
Let $\alpha$ be the Liouville $1$-form defined by $\alpha(x,v).\eta=g(d\pi_0(x,v).\eta,v)$ which satisfies 
$\alpha(X)=1$, $\iota_X d\alpha=0$ and is a contact form. One has 
$\ker \alpha=E_u\oplus E_s$.
Near a point $x_0$ of $M$, one can find isothermal coordinates $x=(x_1,x_2)$ so that the metric is of the form
$g=e^{2\omega(x)}(dx_1^2+dx_2^2)$ for some smooth function $\omega$. Using the action generated by $V$, this induces coordinates $(x_1,x_2,e^{i\theta})$ near the fiber $\pi^{-1}(x_0)$, and in these coordinates 
\[ \begin{gathered}
V=\pl_\theta ,\\
 X= e^{-\omega}(\cos(\theta)\pl_{x_1}+\sin(\theta)\pl_{x_2}+(-\pl_{x_1}\omega\sin\theta+\pl_{x_2}\omega \cos\theta)\pl_\theta),\\
 X_\perp= -e^{-\omega}(-\sin(\theta)\pl_{x_1}+\cos(\theta)\pl_{x_2}-(\pl_{x_1}\omega\cos\theta+\pl_{x_2}\omega \sin\theta)\pl_\theta).
\end{gathered}\]

\textbf{Fourier decomposition in the fibers.} Each smooth function $u\in C^\infty(SM)$ can be decomposed as 
\begin{equation}\label{fourier}
u=\sum_{k\in\zz} u_k, \quad Vu_k=iku_k
\end{equation}
using the Fourier decomposition in the fibers with 
$\|u_k\|_{L^2}=\mc{O}(|k|^{-\infty})$ (see \cite{GK1}) . This decomposition extends to $L^2(SM)$  and induces a splitting 
$L^2(SM)=\oplus_{k\in\zz} H_k$ where elements in $H_k$ correspond to $\ker (V-ik)$ and can be represented as $L^2$-sections of $k$-th power of a complex line bundle over $M$.
The geodesic vector field $X$ acts as a first order differential operator on $C^\infty(SM)$ and can be decomposed as 
\[ X=\eta_++\eta_- , \quad \eta_\pm=\demi(X\pm i X_\perp),\]
where in the decomposition \eqref{fourier}, $\eta_\pm : H_k\to H_{k\pm 1}$ for all $k\in \zz$.
Symmetric cotensors of order $m$ are embedded as functions in $SM$ by the map $\pi_m^*: C^{\infty}(M,\otimes_S^mT^*M)\to C^\infty(SM)$ as before.
The map $\pi_m^*$ is an isomorphism 
\begin{equation}\label{isom}
\pi_m^*: C^\infty(M,\otimes_S^mT^*M)\to (H_m\oplus H_{m-2}\oplus\dots\oplus H_{-m+2}\oplus H_{-m})\cap C^\infty(SM),
\end{equation} 
with sections of $\mc{S}(dz^{m-j}\otimes d\bar{z}^{j})$ mapped to $H_{m-2j}$ by $\pi^*_m$. 
If $q\in H_m$, we have $\pi_m^*{\pi_m}_*q=c_mq$ for some $c_m\not=0$ depending only on $m$. 
The operator $X$ has a splitting   
\[ X=X_++X_-\]
where, using \eqref{isom}, the action on each $H_m$ is given by $X_+=D$ and $X_-=-\demi D^*$.
We define the antipodal map $A: SM\to SM$ by $A(x,v):=(x,-v)$. In terms of coordinates $(x,\theta)$, this is simply the 
translation of $\pi$ in the $\theta$ component. The action by pull-back on functions becomes, 
in the decomposition \eqref{fourier},
\[ A^*:C^\infty(SM)\to C^\infty(SM), \quad A^*u=\sum_{k\in \zz} (-1)^ku_k=\sum_{k\in \zz} u_{2k}-
\sum_{k\in \zz} u_{2k+1}.\]
The kernel of $\demi({\rm Id}-A^*)$ consists of functions with even Fourier coefficients in the fibers,  the 
map $\demi({\rm Id}-A^*)$ has range the set of functions with odd Fourier coefficients.  
These operator extend continuously to $C^{-\infty}(SM)$ by duality since $A$ is a diffeomorphism, and we shall say that a distribution is odd (resp. even) if it is in the range (resp. kernel) of $\demi({\rm Id}-A^*)$. We can write in general $u=u_{\rm ev}+u_{od}$ for distributions, where $u_{od}:=\demi({\rm Id}-A^*)u$.
The operator $X$ maps odd distributions to even distributions and conversely. Therefore one has 
\begin{equation}\label{oddeven}
w\in C^{-\infty}(SM),\,\,  Xw=0 \Longrightarrow Xw_{\rm od}=0\textrm{ and }Xw_{\rm ev}=0.
\end{equation}

\textbf{Szeg\"o projector.} We now define the Szeg\"o projection in the fibers using decomposition \eqref{fourier} 
\[ S: C^\infty(SM)\to C^\infty(SM), \quad Su=\sum_{k\geq 1}u_k\]
which is the projector on the positive Fourier coefficients. This extends as a self-adjoint bounded operator on $L^2(SM)$, and 
as a bounded operator on $H^{s}(SM)$ for all $s\in\rr$.
Moreover an easy computation using $X=\eta_++\eta_-$ gives 
\begin{equation}\label{commut} 
XSu=SXu-\eta_+u_0+\eta_-u_1.
\end{equation}
Let $p_V\in C^\infty(T^*SM)$ be the principal symbol of $-iV$, it is given by $p_V(y,\xi)=\xi(V_y).$
We first recall a standard result.
\begin{lemm}\label{lemm39}
The operator $S$ 
 has Schwartz kernel with wave front set contained in 
\[\{(y,\xi,y,-\xi)\in T^*(SM)\x T^*(SM); \, p_V(\xi)\geq 0\}.\]
\end{lemm}
\begin{proof}
First we notice that the distribution kernel is supported on $\{\pi_0(y)=\pi_0(y')\}$. Then
it suffices to work locally near a point $(x_0,\theta_0,x_0,\theta'_0)\in SM\x SM$ 
and we use isothermal coordinates there. 
The operator $S$ acting on functions supported in a neighborhood of $\pi^{-1}(x_0)$ can be viewed as an operator with a  compactly supported distributional kernel on $(\rr^2\x S^1)\x(\rr^2\x S^1)$; moreover, this is a convolution operator and the convolution kernel is given by
\[ \tilde{S}(x,e^{i\theta})=\delta_0(x)\otimes \frac{1}{2\pi}\sum_{k\geq 1}e^{ik\theta}\] 
The singular support of $\tilde{S}$ is contained in $\{0\}\x S^1$. Notice that $\sum_{k\geq 1}e^{ik\theta}=e^{i\theta}(1-e^{i\theta})^{-1}$ is smooth outside $e^{i\theta}=1$, then for any smooth function $\chi$ on $S^1$ which vanishes near $1$, the wave front set 
of $\chi(e^{i\theta})\tilde{S}$ is contained in the conormal bundle $N^*(\{0\}\x S^1)$, and 
since this was a convolution kernel, when returning to $SM\x SM$ this part has wave front set contained in 
the subset $\{(y,\xi,y,-\xi); p_V(\xi)=0\}$ of the conormal bundle to the diagonal. Now we are left to analyse 
the remaining part of the convolution kernel $(1-\chi(e^{i\theta}))\tilde{S}$. If $1-\chi$ is supported close enough to $1$, we can view this as a distribution on $\rr^2_x\x \rr_\theta$ with Fourier transform 
\[  (\xi_x,\xi_\theta)\mapsto \sum_{k\geq 1}\hat{\psi}(\xi_\theta-k)\] 
where $\psi\in C_0^\infty(\rr)$ is a function equal to $1$ near $0$ and $\psi(\theta)=0$ for $|\theta|>\pi/2$. Since $p_V(\xi_x,\xi_\theta)=\xi_\theta$, 
the Fourier transform decays to all order in all directions $\xi=(\xi_x,\xi_\theta)$ so that $p_V(\xi)<0$. Coming back 
to $SM\x SM$, this part has wave front set contained in $\{(y,\xi,y,-\xi); p_V(\xi)\geq 0\}$.
\end{proof}
 As a corollary of this, we obtain 
\begin{corr}\label{product}
Let $B\in \Psi^0(SM)$ such that ${\rm Id}-B$ is microsupported outside a conic neighborhood of $\{p_V(\xi)=0\}$.
Let $u,v\in C^{-\infty}(SM)$ be odd so that $Bu,Bv\in H^2(SM)$. 
Then the multiplication $w:=S(u).S(v)$ makes sense as a distribution on 
$SM$, $w=S(w)$, and $w=w'+w''$ for some $w'\in H^2(SM)$ and some $w''$
with wave front set ${\rm WF}(w'')\subset \{p_V(\xi)>0\}$. Let $u_1=S(\pi_1^*{\pi_1}_*u)$ and 
$v_1=S(\pi_1^*{\pi_1}_*v)$, then if $Xu=Xv=0$ and $\eta_-u_1=\eta_-v_1=0$, we have $Xw=0$.
\end{corr}
\begin{proof} We write $u=u'+u''$ and $v=v'+v''$ with $u':=Bu$ and 
$v':=Bv$. Then 
$S(u'),S(v')\in H^2(SM)$ so that 
the product $w':=S(u').S(v')$ make sense as an element in $H^2(SM)$ (recall that $H^2(SM)\subset L^\infty(SM)$).
We have ${\rm WF}(u'')\subset \{(y,\xi); |p_V(\xi)|>0\}$ by the microsupport property of $B$, thus we can use \cite[Theorem 8.2.13]{Ho1} and Lemma \ref{lemm39} 
to deduce that ${\rm WF}(S(u''))\subset \{(y,\xi); p_V(\xi)>0\}$ and ${\rm WF}(S(u'))\subset \{(y,\xi); p_V(\xi)\geq 0\}$;  the same property holds for $S(v')$ and $S(v'')$.
Then by \cite[Theorem 8.2.10]{Ho1}, we see that the multiplication $S(u'').S(v)$ and $S(u').S(v'')$ make sense as distributions and have wave front set in $\{(y,\xi); p_V(\xi)>0\}$.  This shows that $S(u).S(v)$ make sense as a distribution and we have the desired result with $w'':=S(u'').S(v)+S(u').S(v'')$.
The fact that $S(u).S(v)=S(S(u).S(v))$ is straightforward to check by taking 
sequences $u_n,v_n\in C^\infty(SM)$ converging to $u,v\in C^{-\infty}(SM)$ as $n\to \infty$ and using that the equality 
$S(S(u_n)S(v_n))$ holds for all $n$.
Moreover, if $u,v$ are odd, in $\ker X$ and $\eta_-u_1=\eta_-v_1=0$,  we deduce from \eqref{commut} that 
$XS(u)=XS(v)=0$ and thus $X(S(u).S(v))=0$, again by approximating by smooth functions in the distribution topology.  This completes the proof.
\end{proof}

\textbf{Proof of Injectivity of X-ray on tensors.} We can now prove 
\begin{theo}
On a Riemannian surface with Anosov geodesic flow, then for all $m\geq 0$ we have 
$\ker I_m\cap \ker D^*=0.$
\end{theo}
\begin{proof} First we claim that for 
 $f_1\in H_1\cap C^\infty(SM)$ which satisfies $\eta_-f_1=0$,  then $\forall s>0$  
 \begin{equation}\label{existv}
 \begin{gathered}
\exists w\in C^{-\infty}(SM) \textrm{ odd }, \, Xw=0, \,\, {\pi_1}_*w={\pi_1}_*f_1, \\
\exists B\in \Psi^0(SM), \,\, {\rm WF}({\rm Id}-B) \cap \{p_V(\xi)=0\}= \emptyset, \,\, Bw\in H^s(SM). 
\end{gathered}\end{equation}  
Indeed,  
the condition $D^*({\pi_1}_*f)=0$ becomes $\eta_+f_{-1}+\eta_-f_1=0$ if $f\in C^\infty(M,T^*M)$ and $\pi_1^*f=f_1+f_{-1}$, so 
if $f_{-1}=0$, the condition $\eta_-f_1=0$ and Corollary \ref{corsurjm} insure that for each $s>0$ there is $w\in \cap_{r<0}H^{r}(SM)$ 
so that $Xw=0$ and ${\pi_1}_*w={\pi_1}_*f_1$, and such that for any $B'\in \Psi^0(SM)$
 with microsupport not intersecting $E_u^*\cup E_s^*$, $B'w\in H^s(SM)$. We have 
that $p_V(\xi)\not=0$ on $(E_u^*\cup E_s^*)\setminus \{0\}$ since $\{p_V(\xi)=0\}$ is the conormal to the fibers of $SM\to M$ (ie. the annihilator of the vertical space of $SM$) and $E_u^*\cup E_s^*$ intersect it only at $\xi=0$, therefore we can choose $B'$ so that ${\rm WF}({\rm Id}-B') \cap \{p_V(\xi)=0\}= \emptyset$ and ${\rm WF}(B')\cap 
(E_u^*\cup E_s^*)=\emptyset$.
Next we take the odd part $w_{\rm od}$ of $w$, then $Xw_{\rm od}=0$ and 
${\pi_1}_*(w_{\rm od})={\pi_1}_*f_1$. If $A$ is the antipodal map, we get $A^*(B'w)\in H^s(SM)$ and  
if $w':=w-B'w$ we have ${\rm WF}(w')\cap \{p_V(\xi)=0\}=\emptyset$ and by \cite[Theorem 8.2.4]{Ho1}, 
\[ {\rm WF}(A^*w')\subset \{(y,(dA)^T\xi); (A(y),\xi)\in {\rm WF}(w')\}\]
But since $dA_y.V_y=V_{A(y)}$ we have 
$p_V(y,(dA_y^T\xi))=p_V(A(y),\xi)$ and thus ${\rm WF}(A^*w')\subset \{p_V(\xi)\not= 0\}$ 
(in fact $A=e^{\pi V}$ and $p_V(\xi)$ is constant under the Hamilton flow of $p_V$). 
Now since $w_{\rm od}=\demi(w'-A^*w')+\demi(B'w-A^*B'w)$, 
we have shown \eqref{existv} by choosing $B\in \Psi^0(SM)$ so that 
${\rm WF}({\rm Id}-B)\cap \{p_V(\xi)=0\}=\emptyset$, 
$B({\rm Id}-B')\in \Psi^{-\infty}(SM)$ and ${\rm WF}(B)\cap {\rm WF}(A^*w')=\emptyset$. Next we show

\begin{prop}\label{nonhyper}
Assume that $M$ is a non hyperelliptic surface with Anosov geodesic flow. Let $f\in C^{\infty}(M,\otimes_S^m T^*M)$ for some $m$, and assume that there exists 
$u\in C^{\infty}(SM)$ with $Xu=\pi_m^*f$.  Then $u=\pi^*_{m-1}q$ for some $q\in C^{\infty}(M,\otimes_S^{m-1}T^*M)$, or equivalently $f=Dq$.
\end{prop}
\begin{proof}
We follow the proof of Theorem 9.3 in \cite{PSU1}. We will do the proof by induction on $m$: let $m\geq 2$ and assume that we have shown that for all $\ell< m$, if $Xu=f$ 
for some $f\in \oplus_{j=0}^{\ell}H_{\ell-2j}$ and $u\in C^\infty(SM)$ 
then $u\in \oplus_{j=0}^{\ell-1}H_{\ell-1-2j}$. This is true for $m= 2$ by the result of \cite{PSU1}. Without loss of generality, let $f\in C^\infty(M,\otimes_S^mT^*M)$ be real valued, that we decompose into irreducible components  
\[\pi^*_mf=f_m+f_{m-2}+\dots+f_{2-m}+f_{-m} \textrm{ with }f_j\in H_j.\]
Assume that there exists $u\in C^\infty(SM)$ real so that $Xu=f$.
We have the orthogonal decomposition 
\[f_m=\eta_+(h_{m-1})+q_m \textrm{ with }h_{m-1}\in H_{m-1}\cap C^\infty(SM), \, q_m\in H_m\cap \ker \eta_-.\]
Then we have $X(u-h_{m-1})=q_m -\eta_-(h_{m-1})+f_{m-2}+\dots +f_{2-m}+f_{-m}$. We use Max Noether's theorem as in \cite[Theorem 9.3]{PSU1}, this implies that $q_m$ is a finite sum $q_m=\sum_{k\in \nn^m} a_{k_1}\dots a_{k_m}$ for some 
$a_{k_j}\in H_1\cap C^\infty(SM)$ and $\eta_-a_{k_j}=0$. By  \eqref{existv}, there 
exists $w_{k_j}\in C^{-\infty}(SM)$ so that $Xw_{k_j}=0$, ${\pi_{1}}_*w_{k_j}=a_{k_j}$ and there is $B\in \Psi^0(M)$
so that ${\rm WF}({\rm Id}-B)\subset \{p_V(\xi)\not=0\}$ and $Bw_{k_j}\in H^2(SM)$ for all $j$. 
By Corollary \ref{product} applied inductively, one can define the product 
$w:=\sum_{k\in \nn^m}S(w_{k_1})\dots S(w_{k_m})$ which satisfies $Xw=0$. By viewing each $w_{k_j}$ as limit of smooth functions, 
$w^{n}_{k_j}$ when $n\to \infty$ we also see that, since the product of $m$ elements $S(w^n_{k_j})$ satisfies 
\[{\pi_{m}}_*(S(w^n_{k_1})\dots S(w^n_{k_m}))={\pi_1}_*(w^n_{k_1})\dots{\pi_1}_*(w^n_{k_m})
\]
then we have by taking the limit $n\to \infty$ that ${\pi_{m}}_*w={\pi_m}_*q_m$, and similarly ${\pi_{\ell}}_*w=0$ for $\ell<m$. 
We then have, using $\pi_m^*{\pi_m}_*q_m=c_mq_m$ and ${\pi_{\ell}}_*w=0$ for $\ell<m$,   
\[ c_m\|q_m\|_{L^2}^2= \cjg \pi_m^*({\pi_{m}}_*w), q_m\cjd=  \cjg w, X(u-h_{m-1})\cjd=-\cjg Xw, u-h_{m-1}\cjd=0\]
and the last equality makes sense since $u-h_{m-1}$ is smooth. Since $f$ is real, $f_{-m}=\bbar{f_m}=\eta_-(\bbar{h_{m-1}})$, and we get 
\[ X(u-h_{m-1}-\bbar{h_{m-1}})=-\eta_+(\bbar{h_{m-1}})-\eta_-(h_{m-1})+f_{m-2}+\dots +f_{2-m}\in \bigoplus_{j=0}^{m-2}H_{m-2-2j}.\]
By the induction assumption, $u-h_{m-1}-\bbar{h_{m-1}}\in \oplus_{j=0}^{m-3}H_{m-3-2j}$ and we have proved the induction. This achieves the proof of the Proposition by using the description \eqref{isom} of symmetric tensors in terms of elements 
in $H_\ell\cap C^\infty(SM)$.
\end{proof}

To conclude the proof of Theorem \ref{injectm}, it suffices to use the argument of the proof of Theorem 1.1. in \cite{PSU1}: $M$ has a normal cover $N$ which is non hyperelliptic and lifting the problem to $N$, we reduce the proof to the case of non hyperelliptic surfaces. Take $f\in C^\infty(M,\otimes_S^mT^*M)$ with $I_m(f)=0$, then Livsic theorem tells us 
that there is $u\in C^{\infty}(SM)$ such that $Xu=\pi_m^*f$. Then Proposition \ref{nonhyper} ends the proof.
\end{proof}

\end{document}